\title{Sequence of families of lattice polarized $K3$ surfaces, modular forms and degrees of complex reflection groups}
\author{Atsuhira Nagano}
\documentclass[10pt]{article}
\usepackage{mathrsfs,bm,graphics,amsfonts,amsthm,amssymb,amsmath,amscd,booktabs}
\usepackage[pdftex]{graphicx}
\usepackage[all]{xy} 
 \usepackage{fancybox}
\def\bigzerou{\smash{\lower1.7ex\hbox{\b 0}}}

\newtheorem{thm}{Theorem}[section]
\newtheorem{df}{Definition}[section]
\newtheorem{lem}{Lemma}[section]
\newtheorem{prop}{Proposition}[section]
\newtheorem{rem}{Remark}[section]

\makeatletter
  
      \@addtoreset{equation}{section}
       
\begin{document}
\maketitle

\begin{abstract}
We introduce a sequence of families of lattice polarized $K3$ surfaces.
This sequence is closely related to complex reflection groups of exceptional type.
Namely, we obtain modular forms coming from the inverse correspondences  of the period mappings attached to our sequence.
We study  a non-trivial relation 
between our modular forms
and invariants of  complex reflection groups.
Especially, 
we consider a family concerned with the Shephard-Todd group  No.34 
based on arithmetic properties of  lattices and algebro-geometric properties of the period mappings.
\end{abstract}

\footnote[0]{Keywords:  $K3$ surfaces ; Modular forms ;  Complex reflection groups ; Compactifications defined by arrangements.  }
\footnote[0]{Mathematics Subject Classification 2020:  Primary 14J28 ; Secondary   11F11,  20F55, 32S22.}

\section*{Introduction}
In the 20th century,
Brieskorn  founded 
an interesting theory which connects finite real reflection groups, Klein singularities and families of rational surfaces  (see \cite{Br}).
For example,
according to this theory,
a family of rational surfaces defined by the equation
\begin{align}\label{SurfaceE8}
z^2=y^3+( \alpha_2 x^3+ \alpha_{8} x^2 + \alpha_{14} x+ \alpha_{20})y+(x^5+ \alpha_{12} x^3 + \alpha_{18} x^2 + \alpha_{24} x +\alpha_{30})
\end{align}
is characterized by the real reflection group $W(E_8)$.
Namely, the theory enables us 
to interpret 
the parameters $\alpha_2,\alpha_{8},\alpha_{12},\alpha_{14},\alpha_{18},\alpha_{20},\alpha_{24}$ and $\alpha_{30}$ as the invariants of $W(E_8)$ (see \cite{Sl} Chapter IV or \cite{H} Chapter 5).
Many researchers have been attracted by this theory and
they have tried to generalize it.
Indeed,
Arnold suggested that
it is an interesting problem to obtain an analogous theory for finite complex  reflection groups (see \cite{A} p.20).
There are many works for that problem based on various ideas, viewpoints and techniques.

On the other hand,
the author has studied modular forms derived from periods of $K3$ surfaces
and realized a potential of those modular forms to be applied to
studies for complex reflection groups.
In this paper,
we introduce a sequence of  families of $K3$ surfaces
whose period mappings are closely related to  complex reflection groups.

Let $U$ be the even hyperbolic lattice of rank $2$.
Let $A_m$ or $E_m$ be the root lattices of rank $m$. 
Then, the $K3$ lattice $L_{K3}$ is given by  $U^{\oplus 3}\oplus E_8(-1)^{\oplus 2}$.
Here, if a lattice $\Lambda$  with the intersection matrix $(c_{ij})$ is given, the lattice given by $(n c_{ij})$ is denoted by $\Lambda (n)$.
The $2$-homology group of a $K3$ surface $S $ is isometric to $L_{K3}$.
The N\'eron-Severi lattice, which is denoted by ${\rm NS}(S)$, is a sublattice of $H_2(S,\mathbb{Z})$ of signature $(1,\rho-1)$,
if $S$ is an algebraic $K3$ surface.
The transcendental lattice ${\rm Tr}(S)$ is the orthogonal complement of ${\rm NS}(S)$ in $H_2(S,\mathbb{Z})$.
Then, ${\rm Tr}(S)$ is an even lattice of signature $(2,20-\rho)$.

In this paper, 
we study a sequence of even lattices 
$$
{\bf A}_3 \subset {\bf A}_2 \subset {\bf A}_1 \subset {\bf A}_0 ={\bf A}, 
$$
where
\begin{align}\label{LatticesAj}
\begin{cases}
&   {\bf A}_0 =U\oplus U\oplus A_2(-1) \oplus A_1(-1),\\
& {\bf A}_1=U\oplus U\oplus A_2(-1)  ,\\
& {\bf A}_2= U\oplus U\oplus A_1(-1) ,\\
& {\bf A}_3= U\oplus \begin{pmatrix} 2 & 1 \\ 1 & -2 \end{pmatrix}.
\end{cases}
\end{align}
Set 
${\bf M}_j={\bf A}_j^\perp$ in $L_{K3}$.
Then, ${\bf A}_j$ is of type $(2,5-j)$ and ${\bf M}_j$ is of type $(1,14+j).$
We introduce a sequence of analytic sets
$$
\mathfrak{A}_3 \subset \mathfrak{A}_2 \subset \mathfrak{A}_1 \subset \mathfrak{A}_0=\mathfrak{A}
$$
with ${\rm dim}(\mathfrak{A}_j)=7-j$
(for detail, see Lemma \ref{LemA} and (\ref{ParameterSubA})).
Then, we  obtain an explicit family
\begin{align}\label{varpi}
\varpi_j: \mathfrak{F}_j \rightarrow \mathfrak{A}_j \quad (j\in\{0,1,2,3\})
\end{align}
of  ${\bf M}_j$-polarized $K3$ surfaces.
This roughly means that the N\'eron-Severi lattice of a generic member of $\mathfrak{F}_j$ is ${\bf M}_j.$
These families (\ref{varpi}) constitute  a sequence of families of  $K3$ surfaces indicated in the following diagram:
\begin{align}\label{diagram1}
  \xymatrix@C=50pt{
\mathfrak{F}_3 \ar@{^{(}-_>}[r]^{\tilde{i}_3}  \ar[d]_{\varpi_3}  & \mathfrak{F}_2 \ar@{^{(}-_>}[r]^{\tilde{i}_2} \ar[d]_{\varpi_2} &  \mathfrak{F}_1 \ar@{^{(}-_>}[r]^{\tilde{i}_1} \ar[d]_{\varpi_1}  & \mathfrak{F}_0  \ar[d]_{\varpi_0}  \\
\mathfrak{A}_3 \ar@{^{(}-_>}[r]^{i_3} &\mathfrak{A}_2 \ar@{^{(}-_>}[r]^{i_2} & \mathfrak{A}_1 \ar@{^{(}-_>}[r]^{i_1}&  \mathfrak{A}_0 \\
}
\end{align}
Here, $i_j$ and $\tilde{i}_j$ are natural inclusion.
The period domain for $\mathfrak{F}_j$ is given by a connected component $\mathcal{D}_j$ of 
\begin{align}\label{DM}
\mathcal{D}_{{\bf M}_j} =\left\{[\xi]\in \mathbb{P} ({\bf A}_j \otimes \mathbb{C} ) \mid {}^t \xi {\bf A}_j \xi =0, {}^t \xi {\bf A}_j \overline{\xi} >0 \right\},
\end{align}
which is $(5-j)$-dimensional.

The main theme of the present paper is
to show that 
there is an interesting and non-trivial relation between the sequence (\ref{diagram1}) and finite complex reflection groups of exceptional type.

The first purpose of this paper is 
to study the period mapping of the family $\mathfrak{F}_0$.
The inverse of the period mapping gives a pair of meromorphic modular forms on $\mathcal{D}_0$ (see Definition \ref{DfMeroModularForm}).
We will obtain a system of generators of the ring of these modular forms  by applying techniques for periods of $K3$ surfaces (Theorem \ref{ThmId} and \ref{Thmdet}).
In short,
we study a family of $K3$ surfaces defined by the equation
\begin{align}\label{SurfaceCRG}
z^2 = y^3 +(a_0 x^5 + a_4 x^4 + a_8 x^3 ) y +(a_2 x^7  + a_6 x^6 + a_{10} x^5  + a_{14} x^4 )
\end{align}
and we show that the parameters $a_2, a_4,a_6,a_8,a_{10}$ and $a_{14}$ with positive weight induce a system of generators of the ring of modular forms.

Our modular forms are highly expected to have a closed relation with the complex reflection group  No.34 in the list of Shephard-Todd \cite{ST} (see also \cite{LT} Appendix D),
because three times the weights of the modular forms (namely, $6,12,18,24,30$ and $42$) are equal to the degrees of the group.
This group has the maximal rank among finite complex reflection groups of exceptional type.
This expectation is based on 
not only the above mentioned apparent similarity between the weights and the degrees,
but also the following fact.
There are exact descriptions 
of  
the period mappings for the subfamilies $\mathfrak{F}_1,\mathfrak{F}_2,\mathfrak{F}_3$ 
via  invariants of  complex reflection groups.
Precisely,
the period mappings for $\mathfrak{F}_1$ ($\mathfrak{F}_2, \mathfrak{F}_3$, resp.)
derive Hermitian (Siegel, Hilbert, resp.) modular forms
with explicit expressions 
via the invariants of the complex reflection group No.33 (No.31, No.23, resp.) of rank $r_j$
and a system of appropriate theta functions
(for detail, see Section 4.3; see also Remark \ref{RemKappa}).
Let $\left(w_1^{(j)},\ldots,w_{r}^{(j)}\right)$ be the weights of the modular forms  as in Table 1.
Then, the degrees of the complex reflection groups are given by $\left(\kappa_j w_1^{(j)},\ldots, \kappa_j w_{r}^{(j)}\right).$
Here, the weights of the theta functions account for  the integer $\kappa_j$. 
\begin{table}[h]
\center
\vspace{-3mm}
\caption{Modular forms coming from $K3$ surfaces and complex reflection groups}
\vspace{1mm}
\begin{tabular}{ccccccc}
\toprule
  $j$ &Modular forms     & Weights & $K3$ surfaces  & Reflection groups & $r_j$ & $\kappa_j$ \\
 \midrule
$0$ & see Definition \ref{DfMeroModularForm} & $2,4,6,8,10,14$ & This paper &  No.34 & $6 $ & $3$ \\
$1$ &Hermitian   & $4,6,10,12,18$  & \cite{NS1}& No.33 & $5$ & $1$ \\
$2$ &Siegel &  $4,6,10,12$  & \cite{CD} & No.31 & $4$ & $2$ \\
$3$ &Hilbert &   $2,6,10$  & \cite{NHilb}  & No.23 & $3$ & $1$ \\
 \bottomrule
\end{tabular}
\end{table} 

In these works for the families $\mathfrak{F}_j$ $(j\in\{1,2,3\})$, 
the Satake-Baily-Borel compactifications for bounded symmetric domains play big roles.
However, 
in the case of $\mathfrak{F}_0$,
the Satake-Baily-Borel compactifications are inadequate,
because we need to consider modular forms defined on a complement of an arrangement of hyperplanes (see Section 3 and 5).
So, 
instead of the Satake-Baily-Borel compactifications,
we will consider the Looijenga compactifications  constructed in \cite{L}.
The Looijenga compactifications are coming from  arithmetic arrangements of hyperplanes.
We can regard them as  interpolations between the Satake-Baily-Borel compactifications and  toroidal compactifications.
Their properties  are essential for our construction of modular forms.

By the way,
there exists a double covering of every member of $\mathfrak{F}_0$,
which is a $K3$ surface also.
We obtain the family
$$
\overline{\varpi_j}:\mathfrak{G}_j \rightarrow \mathfrak{A}_j
\quad (j\in\{0,1,2,3\})
$$
whose members are such double coverings
and we have the following diagram:
 \begin{align}\label{TriDiagram}
\xymatrix{
\mathfrak{G}_j\ar[rr]^-{\varphi_j}\ar[dr]_-{\overline{\varpi_j}}&\ar@{}@<0.8ex>[d]|{}&\mathfrak{F}_j\ar[dl]^-{\varpi_j}\\
&\mathfrak{A}_j&
}
\end{align}
Here, $\varphi_j$
is a correspondence given by the double covering.

The second purpose of this paper is
to determine the transcendental lattice for $\mathfrak{G}_0.$
The family $\mathfrak{G}_0$ has interesting features.
For example,
it naturally contains the famous family of Kummer surfaces
coming from principally polarized Abelian surfaces.
Moreover,
$\mathfrak{G}_0$ is a natural extension of the family studied by Matsumoto-Sasaki-Yoshida \cite{MSY},
whose periods are solutions of the hypergeometric equation of type $(3,6)$.
In spite of interesting properties of $\mathfrak{G}_j$ $(j\in\{0,1,2,3\})$,  it is not straightforward to determine the lattices for them.
For example, if $j\in \{1,2,3\}$,
the lattices for $\mathfrak{G}_j$ were determined via precise arguments or heavy calculations
(for detail, see Section 6).
In the present paper,
we will determine the transcendental lattice for $\mathfrak{G}_0,$
based on the result of $\mathfrak{G}_1$ and arithmetic properties of even lattices (Theorem \ref{ThmTrK}).
As a result,
the transcendental lattices ${\bf B}_j$ for $\mathfrak{G}_j$ are given as follows.
\begin{align}\label{LatticesBj}
\begin{cases}
&   {\bf B}_0 =U(2)\oplus U(2)\oplus \begin{pmatrix} -2& 0 & 1 \\ 0& -2 & 1 \\  1& 1 & -4 \end{pmatrix},\\
& {\bf B}_1= {\bf A}_1(2) =U(2)\oplus U(2)\oplus A_2(-2)  ,\\
& {\bf B}_2={\bf A}_2 (2) =U(2)\oplus U(2)\oplus A_1(-2) ,\\
& {\bf B}_3={\bf A}_3 (2) =U(2)\oplus \begin{pmatrix} 4 & 2 \\ 2 & -4 \end{pmatrix}.
\end{cases}
\end{align}
Especially,
we note that ${\bf B}_0$ is not just  ${\bf A}_0(2).$

It is an interesting problem to describe our meromorphic modular forms
via the invariants of the group No.34 and explicit special functions (like theta functions).
Furthermore,
it may be quite meaningful  to understand
why the period mappings for the sequence (\ref{diagram1}) are related with complex reflection groups. 
While the methods in this paper  
and preceding papers \cite{CD}, \cite{NHilb} and \cite{NS1} are just based on
algebro-geometric properties of $K3$ surfaces
and arithmetic properties of  modular forms,
the author does not know the fundamental reason why the complex reflection groups work effectively as in Table 1.
The author expects that 
there exists an unrevealed principle 
underlying the relation between our sequence of the families and complex reflection groups.
There are indications which support this expectation.
For example,
Sekiguchi \cite{S} studies Arnold's problem
based on methods of Frobenius potentials
and he obtains a family of rational surfaces.
Although his standpoint and methods are widely different from ours,  
a direct calculation shows that
 our $K3$ surface of (\ref{SurfaceCRG}) is related to Sekiguchi's rational surface  (see Remark \ref{RemSekiguchi}).  
The author is hoping that 
a new principle will rationalize our families of $K3$ surfaces  from the viewpoint of complex reflection groups in near future,
as Brieskorn's theory enables us to explain the essence of the family of the rational surfaces of  (\ref{SurfaceE8}).

\section{Arithmetic arrangement of hyperplanes and  Looijenga compactification}

Looijenga constructed  compactifications for  bounded symmetric domains of type $IV$ derived from  arithmetic arrangements of hyperplanes.
First, we will survey his result.
For detail, see \cite{L}.

Let $V$ be an $(n+2)$-dimensional vector space over $\mathbb{C}$
with a non-degenerated symmetric bilinear form $\varphi.$
Now, we suppose that $(V,\varphi)$ has been defined over $\mathbb{Q}$ and $\varphi$ is of signature $(2,n)$ for the $\mathbb{Q}$-structure.
Then, the set 
$\{[v]\in V \mid \varphi(v,v)=0, \varphi(v,\overline{v})>0\}$ has two connected components $\mathscr{D} $ and $\mathscr{D}_-$.
We take $\mathscr{D}$ from $\{\mathscr{D} ,\mathscr{D}_-\}$.
 This is a bounded symmetric domain of type $IV$.
For a linear subspace $L$ of $V$,
we set $\mathscr{D}_L = \mathscr{D} \cap \mathbb{P}(L).$
The orthogonal group $O(\varphi)$ is an algebraic group over $\mathbb{Q}$.
Let $\Gamma$ be an arithmetic subgroup of $O(\varphi)$.
We set $X=\mathscr{D}/\Gamma$.

If a hyperplane $H$ of $V$
is defined over $\mathbb{Q}$ and of signature $(2,n-1)$,
it gives a hypersurface $\mathscr{D}_H\not= \phi.$
Suppose $\mathscr{H}$ is a $\Gamma$-invariant arrangement of hyperplanes  satisfying this property.
Such an arrangement is said to be arithmetic
if it is given by a finite union of $\Gamma$-orbits.

Set
$$
\mathscr{D}^\circ = \mathscr{D} - \bigcup_{H\in \mathscr{H}} \mathscr{D}_H.
$$
Looijenga \cite{L} constructs a natural compactification of $X^\circ=\mathscr{D}^\circ/\Gamma$.
Namely, the Looijenga compactification is given  by
\begin{align}\label{LCompact}
\widehat{X^\circ}^{\bf L}:= \widehat{\mathscr{D}^\circ}^{\bf L}/\Gamma,
\end{align}
where
\begin{align}\label{LD}
\widehat{\mathscr{D}^\circ}^{\bf L}=\mathscr{D}^\circ \sqcup \coprod_{L\in {\bf PO}(\mathscr{H}|_\mathscr{D})} \pi_L (\mathscr{D}^\circ) \sqcup \coprod_{\sigma\in \Sigma(\mathscr{H})} \pi_{\sigma} (\mathscr{D}^\circ).
\end{align}
The disjoint union of (\ref{LD}) admits an appropriate topology
and $\mathscr{D}^\circ$ is an open and dense set in $\widehat{\mathscr{D}^\circ}^{\bf L}.$
We remark that
the Looijenga compactification $\widehat{X^\circ}^{\bf L}$ coincides with the Satake-Baily-Borel compactification $\widehat{X}^{\bf SBB}$ when $\mathscr{H}=\phi.$
In the following, we will see
the meaning of (\ref{LCompact}) and (\ref{LD}).

Letting $L$ be a subspace of $V$,
we have a natural projection 
$\pi_L: \mathbb{P}(V)-\mathbb{P}(L)\rightarrow \mathbb{P}(V/L)$.
Let $ {\bf PO}(\mathscr{H}|_\mathscr{D})$  be
a set  of subspaces $L$ of $V$ 
such that 
there exists $z\in L$ with $\varphi(z,\overline{z})>0$.
We note that $L\in {\bf PO}(\mathscr{H}|_\mathscr{D})$ 
if and only if $\mathscr{D}_L\not=\phi.$
Also,
we remark that ${\bf PO}(\mathscr{H}|_\mathscr{D})$ is a partially ordered set (see \cite{LBall} Section 2 and 3).

For a $\mathbb{Q}$-isotropic line $I$ in $V$,
$\varphi$ defines a bilinear form on the $n$-dimensional space $I^\perp/I$ of signature $(1,n-1)$. 
Taking  the choice of $\mathscr{D}$ from $\{\mathscr{D},\mathscr{D}_-\}$ into account, 
we have an $n$-dimensional cone $C_I$ in  $I^\perp/I$.
Let $C_{I,+}(\subset I^\perp /I)$ be the convex hull of $(I^\perp/I) \cap \overline{C_I}$.
If a member $H\in \mathscr{H}$ contains $I$,
then it naturally determines a hyperplane $H_{I^\perp/I}$ of $I^\perp/I$ of signature $(1,n-2).$ 
The hyperplanes $H_{I^\perp/I}$ with $H_{I^\perp/I}\cap C_{I,+}\not=\phi$ 
give a decomposition $\Sigma(\mathscr{H})_I$ of the cone $C_{I,+}$ into locally rational cones.

We set
$\Sigma(\mathscr{H})=\bigcup_{I} \Sigma(\mathscr{H})_I$. 
For a cone
$\sigma \in \Sigma(\mathscr{H})_I (\subset \Sigma(\mathscr{H})),$ 
we have the $\Sigma$-support space $V_\sigma (\subset I^\perp)$.
 Namely, $V_\sigma$ contains $I$ 
 and corresponds to the $\mathbb{C}$-span of the cone
$\sigma(\subset I^\perp/I)$.
Hence,
$V_\sigma$ is given by  the intersection of $I^\perp$ and the members $H$ of $\mathscr{H}$ such that $H\supset I$. 
We put
$\pi_\sigma=\pi_{V_\sigma}$.

Set
\begin{align}\label{DSigmaH}
\mathscr{D}^{\Sigma(\mathscr{H})} = \coprod_{\sigma \in \Sigma(\mathscr{H})} \pi_\sigma (\mathscr{D}).
\end{align}
Then, $X^{\Sigma(\mathscr{H})} =\mathscr{D}^{\Sigma(\mathscr{H})} / \Gamma$ 
is a normal analytic space.

We have a blowing up
$\widetilde{X^\circ} \rightarrow X^{\Sigma (\mathscr{H})}$,
which is coming from the connected components of intersections of members  $H\in \mathscr{H}$.
The Looijenga compactification $\widehat{X^\circ}^{\bf L}$ of (\ref{LCompact}) is equal to a  blowing down 
$\widetilde{X^\circ} \rightarrow \widehat{X^\circ}^{\bf L}.$

\begin{thm}(\cite{L} Corollary 7.5)\label{ThmLooijenga}
Suppose that every $\pi_\sigma (\mathscr{D})$ in (\ref{DSigmaH}) is not  $(n-1)$-dimensional.
Then, the algebra 
$$
\bigoplus_{k\in \mathbb{Z}} H^0(\mathscr{D}^\circ,\mathcal{O}(\mathscr{L}^k))^\Gamma,
$$
where $\mathscr{L}$ is the natural automorphic bundle over $\mathscr{D}$,
is finitely generated with positive degree generators.
Its {\rm Proj} gives the Looijenga compactification $\widehat{X^\circ}^{\bf L}$ of (\ref{LCompact}). 
The boundary $\widehat{X^\circ}^{\bf L}-X^\circ$ is the strict transform of the boundary
$X^{\Sigma (\mathscr{H})} - X$.
Especially,
$$
{\rm codim} \left( \widehat{X^\circ }^{\bf L} - X^\circ \right) 
\geq 2
$$
holds.
\end{thm}

Let $\widehat{\mathscr{L}}$ be an ample line bundle on $\widehat{X^\circ}^{\bf L}$ such that $\widehat{\mathscr{L}}|_{X^\circ}=\mathscr{L}|_{X^\circ}$.
It is shown in \cite{L} that
every meromorphic  $\Gamma$-invariant automorphic form whose poles are contained in $\mathscr{H}$ 
is corresponding to a meromorphic section $s$ of $\widehat{\mathscr{L}}$ such that $s|_{X^\circ}$ is holomorphic.

\subsection{Lattice ${\bf A}$ and arrangement $\mathcal{H}$}
We set
\begin{align}\label{LatticeA}
{\bf A}={\bf A}_0=U\oplus U\oplus A_2(-1) \oplus A_1(-1).
\end{align}
For this lattice ${\bf A}$, we set
\begin{align}\label{Gamma}
\Gamma = \tilde{O}({\bf A}) \cap O^+({\bf A}).
\end{align}
Here,
$\tilde{O}({\bf A})$ is the stable orthogonal group:
$\tilde{O}({\bf A}) = {\rm Ker}\left(O({\bf A}) \rightarrow {\rm Aut}({\bf A}^\vee/{\bf A}) \right)$,
where ${\bf A}^\vee={\rm Hom}({\bf A},\mathbb{Z}).$ 
Also, $O^+({\bf A})$ is the subgroup of $O({\bf A})$ which preserves the connected component $\mathcal{D}$.
The lattice ${\bf A}$ satisfies the Kneser conditions in the sense of  Gritsenko-Hulek-Sankaran \cite{GHS}.
Therefore, we have the following result.

\begin{prop}\label{PropKneser}
Let $\Delta({\bf A})$ be the set of vectors $v\in {\bf A}$ such that $(v \cdot v)=-2$.
The group $\Gamma$ of (\ref{Gamma}) is generated by reflections $\sigma_\delta: z \mapsto z+(z\cdot \delta)\delta$ for $\delta \in \Delta({\bf A})$
and it holds
${\rm Char}(\Gamma)=\{{\rm id}, {\rm det}\}$.
\end{prop}

Here, we note that
the intersection number of elements $v_1$ and $v_2$ of lattices are often  denoted by $(v_1\cdot v_2)$
in this paper.

\begin{lem}\label{LemGamma}
The group $\Gamma$ is isomorphic to the projective orthogonal group $PO^+({\bf A})$.
\end{lem}

\begin{proof}
Let $\{\alpha_1,\alpha_2,\alpha_3\}$ be a basis of $A_2(-1) \oplus A_1(-1)$ with $(\alpha_j \cdot \alpha_j)=-2$ $(j\in \{1,2,3\})$, $(\alpha_1\cdot \alpha_2)=1$ and $(\alpha_k \cdot \alpha_3)=0$ $(k\in \{1,2\})$.  
Then, 
$y_1=\frac{1}{3} \alpha_1 + \frac{2}{3} \alpha_2, y_2=\frac{2}{3} \alpha_1 +\frac{1}{3} \alpha_2$ and $y_3=\frac{1}{2}\alpha_3$
generate the discriminant group ${\bf A}^\vee/ {\bf A}$.
It follows that
$-id_{O^+({\bf A})}\not\in \Gamma$ 
and
$O^+({\bf A})/\Gamma \simeq \mathbb{Z}/2\mathbb{Z}$.
Hence, the assertion follows.
\end{proof}

We consider the case
$V=  {\bf A}\otimes \mathbb{C}$.
Let $e_j,f_j$ $(j\in\{1,2\})$ be elements of $U^{\oplus 2}$ satisfying $(e_j\cdot e_k)=(f_j\cdot f_k)=0$ and $(e_j\cdot f_k)=\delta_{j,k}$.
Let $\{\alpha_1,\alpha_2,\alpha_3\}$ be the  basis of $A_2(-1)\oplus A_1(-1)$  as in Lemma \ref{LemGamma}. 
A vector $v\in V$ is given by the form
\begin{align}\label{VectorA}
v= \xi_1 e_1 + \xi_2 f_1 + \xi_3 e_2 + \xi_4 f_2 + \xi_5 \alpha_1 + \xi_6 \alpha_2 + \xi_7 \alpha_3.
\end{align}
Let us consider a hyperplane $H_0=\{\xi_7=0\}$ in $V$.
For $\Gamma$ of (\ref{Gamma}), the $\Gamma$-orbits of
$H_0$ give an arithmetic  arrangement $\mathcal{H}$.

\begin{lem}\label{LemArrangement}
The above  arrangement $\mathcal{H}$ of hyperplanes satisfies the condition of
Theorem \ref{ThmLooijenga}.
\end{lem}

\begin{proof}
We will prove that all members of $\mathcal{H}$ contain a non-zero common subspace of the negative-definite vector space $\langle \alpha_1,\alpha_2,\alpha_3\rangle_\mathbb{C}$. 
Then,
it is guaranteed that
our arrangement $\mathcal{H}$ satisfies the condition of Theorem \ref{ThmLooijenga},
as  in the argument of \cite{LBall} Section 6.

The hyperplane $H_0=\{\xi_7 =0\}$ of $V$ is generated by 
the basis $\{e_1,f_1,e_2,f_2,\alpha_1,\alpha_2\}$.
Using the notation in the proof of Lemma \ref{LemGamma}, 
${\bf A}^\vee /{\bf A}$ is generated by  $y_1,y_2$ and $y_3$.
We note that every $\gamma\in \Gamma$ 
fixes $y_j \in {\bf A}^\vee/ {\bf A}$.
This implies that 
we can take a basis of the subspace $\gamma H_0$ 
such that this basis is an extension of $\{\alpha_1,\alpha_2\}$.
Therefore,
every member of $\mathcal{H}$ contains the $2$-dimensional  subspace $\langle \alpha_1,\alpha_2 \rangle_\mathbb{C}$ of $\langle \alpha_1,\alpha_2,\alpha_3\rangle_\mathbb{C}$.
\end{proof}

\begin{rem}
The condition of Theorem \ref{ThmLooijenga} is not always satisfied.
For example, 
under the notation (\ref{VectorA}),
let us take a hyperplane $H_0'=\{\xi_5-\xi_6=0\}$.
We can see that the arrangement $\mathcal{H}'$ of $\gamma H_0'$ for $\gamma \in \Gamma$
 does not satisfy the condition of Theorem \ref{ThmLooijenga}.
 This condition is closely related to whether an arrangement gives the zero of a modular form or not.
 In fact, we can see that
 $\mathcal{H}$ is not the zero set of a modular form, whereas $\mathcal{H}'$ does.  
\end{rem}

For our symmetric space $\mathcal{D}=\mathcal{D}_0$, which is a connected component of $\mathcal{D}_{{\bf M}_0}$ of (\ref{DM}), set 
\begin{align}\label{Dcirc}
\mathcal{D}^\circ=\mathcal{D}-\bigcup_{H \in \mathcal{H}} \mathcal{D}_H.
\end{align} 
Due to Theorem \ref{ThmLooijenga} and Lemma \ref{LemArrangement},
we have the following result for our arrangement $\mathcal{H}$.

\begin{prop}\label{PropCodim}
$$
{\rm codim}\left(\widehat{\mathcal{D}^\circ/\Gamma}^{\bf L}  - \mathcal{D}^\circ/\Gamma \right)\geq 2.
$$
\end{prop}

\section{Family $\mathfrak{F}_0$ of $K3 $ surfaces with Picard number $15$}

In the present paper,
we will  consider the Looijenga compactification coming from the arithmetic arrangement $\mathcal{H}$ in Section 1.1.
In order to obtain an explicit model of  the compactification, 
we will introduce a family of elliptic $K3$ surfaces whose transcendental lattice is ${\bf A}$ in (\ref{LatticeA}) (see Theorem \ref{TheoremLattice}).
Periods for $K3$ surfaces are very important in our argument.
We remark that
the period mapping for a family  of $K3$ surfaces is essentially related to the arithmetic property of the transcendental lattice for a generic member of the family.

For $a=(a_0,a_2,a_4,a_6,a_8,a_{10},a_{14})\in \mathbb{C}^7-\{0\}=: \mathbb{C}_a$, we consider the hypersurface $S_a$ defined by an equation 
\begin{align}\label{SK3}
S_a : z^2 = y^3 +(a_0 x^5 + a_4 x^4 w^4+ a_8 x^3 w^8) y +(a_2 x^7 w^2 + a_6 x^6 w^6 + a_{10} x^5 w^{10} + a_{14} x^4 w^{14})
\end{align}
of weight $30$ in the weighted projective space ${\rm Proj}(\mathbb{C}[x,y,z,w])=\mathbb{P}(4,10,15,1)$.
We have a natural action of $\mathbb{C}^*$
on $\mathbb{P}(4,10,15,1)$
given by $(x,y,z,w)\mapsto (x,y,z,\lambda^{-1}w)$
and that on $\mathbb{C}_a$ given by $a \mapsto \lambda\cdot a  =(\lambda^k a_k)=(a_0, \lambda^2 a_2, \lambda^4 a_4, \lambda^6 a_6, \lambda^8 a_8, \lambda^{10} a_{10}, \lambda^{14} a_{14})$.

By applying \cite{NSA} Proposition 3.1, we have the following result.

\begin{lem}\label{LemA}
Let $\mathfrak{A}_0=\mathfrak{A}$ be the set of parameters $a\in  \mathbb{C}_a$ such that $S_a$ is a $K3$ surface.
Then, $\mathfrak{A}_0$ is a subset of
$
 \{a\in \mathbb{C}_a \mid a_0 \not=0\} \cup \{a\in \mathbb{C}_a \mid a_2\not=0\}
$
such that
$$
\mathbb{C}_a-\mathfrak{A}_0= \mathcal{C}' \sqcup \mathcal{C}''
$$
where
\begin{align*}
\begin{cases}
\mathcal{C}' & = \{a\in  \mathbb{C}_a \mid a_0\not=0,  a_{10}=a_{12}=a_{18}=0\} \subset \{a_0\not=0\},\\
\mathcal{C}''& = \{a\in \mathbb{C}_a  \mid a_2\not=0, a_0=a_{10}=a_{12}=a_{18}=0 \} \subset \{a_2\not=0\}.
\end{cases}
\end{align*}  
\end{lem}

\begin{rem}
The surface $S_a$ degenerates to a rational surface if $a_0=a_2=0.$
\end{rem}

We have a family
$$
\varpi_0: \mathfrak{F}_0 =\{S_a \text{ of } (\ref{SK3}) \mid a \in \mathfrak{A}_0\} \rightarrow \mathfrak{A}_0
 $$
of elliptic $K3$ surfaces.

\subsection{Singular fibres}

The Weierstrass equation (\ref{SK3}) defines an elliptic surface $\pi_a: S_a\rightarrow \mathbb{P}^1(\mathbb{C}).$
For a generic point $a\in \mathfrak{A}_0$, we have singular fibres for  $\pi_a$ of Kodaira type 
\begin{align}\label{SingularFibre}
III^* + IV^* + 7 I_1,
\end{align}
as illustrated in Figure 1.
Now, $\pi_a^{-1}(\infty)$ ($\pi_a^{-1}(0)$, resp.) is a singular fibre  of Kodaira type $III^*$ ($IV^*$, resp.). 
Each gives an  $E_7$-singularity and an $E_6$-singularity, respectively.

Set $x_0=\frac{x}{w^4}$ and
\begin{align*}
 g_2^\vee (x_0,a)= a_0 x_0^5 + a_4 x_0^4 + a_8 x_0^3, 
 \quad
g_3^\vee(x_0,a)= a_2 x_0^7 + a_6 x_0^6 + a_{10} x_0^5 + a_{14} x_0^4.
\end{align*}
Let $r(a)$ be the resultant of $g_2^\vee (x_0,a)$ and $g_3^\vee (x_0,a)$ in $x_0$.
Also, let $R(x_0,a)$ be a polynomial in $x_0$ coming from the discriminant  of the right hand side of (\ref{SK3}) in $y$.
So, the discriminant of $\frac{1}{x_0^8} R(x_0,a)$ in $x_0$ is given by $r(a)^3 d_{84}(a) $,
where $d_{84} (a)$ can be calculated as a polynomial in $a$ of weight $84$.

\begin{figure}[h]
\center
\includegraphics[scale=0.5, bb=200 270 500 550]{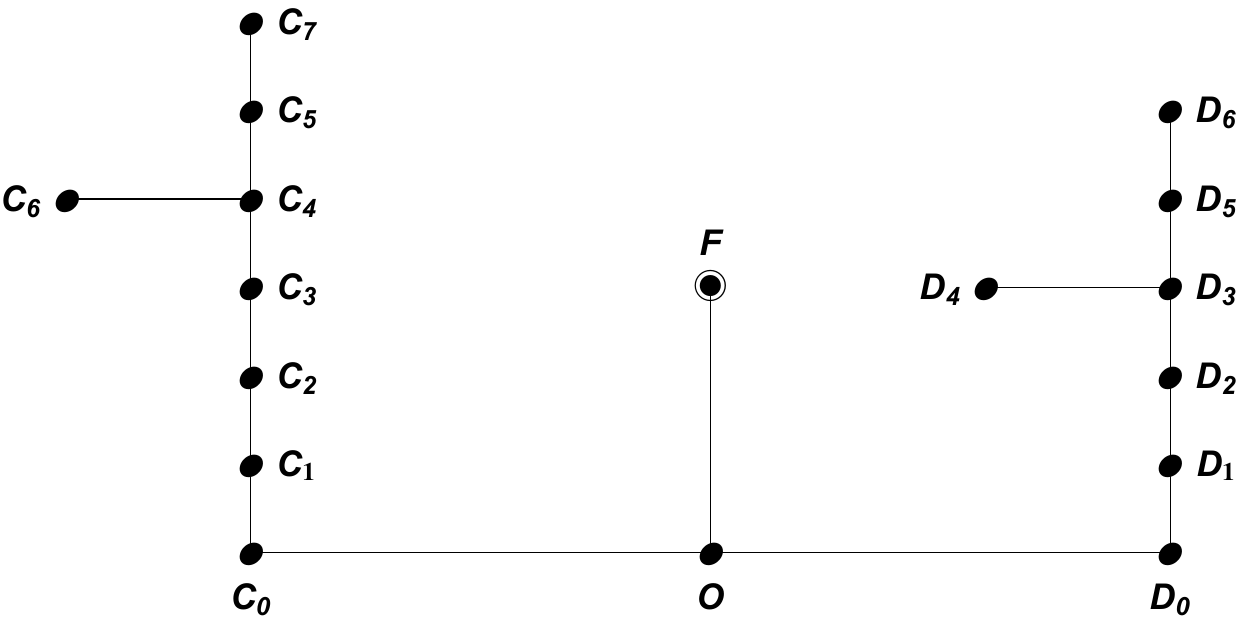}
\caption{Singular fibres for $\pi_a: S_a\rightarrow \mathbb{P}^1(\mathbb{C})$}
\end{figure}

The following lemma can be proved by arguments of elliptic surfaces as in \cite{NSA} Section 3,  \cite{HU} Section 6 and \cite{Na} Section 1.
In particular, we can determine the Kodaira type of each singular fibre by \cite{NSA} Remark 3.5.

\begin{lem}\label{LemSingular}
Except for the generic case corresponding to the fibres of (\ref{SingularFibre}),
the types of the singular fibres for the elliptic surface $\pi_a:S_a \rightarrow \mathbb{P}^1(\mathbb{C}) $ $(a\in \mathfrak{A}_0)$ are given by the following.

\begin{itemize}
\item
If $a\in \mathfrak{A}_0$ satisfies $r(a)=0$,
there is a new singular fibre of Kodaira type $II$ on the elliptic surface $S_a$.
Such a singular fibre does not acquire any new singularities.

\item
If $a\in \mathfrak{A}_0$ satisfies $d_{84}(a)=0,$
two of singular fibres of Kodaira type $I_1$ in (\ref{SingularFibre}) collapse into a singular fibre of type $I_2$:
$$
III^* + IV^* + I_2 + 5 I_1
$$
In this case,
a new $A_1$-singularity appears on the $K3$ surface $S_a$. 

\item 
If $a_0=0$, there are the singular fibres of Kodaira type
$$
II^* + IV^* +6 I_1
$$
on $S_a$.
In this situation, the $E_7$-singularity of (\ref{SingularFibre}) turns into  an $E_8$-singularity.

\item
If $a_{14}=0$, there are the singular fibres of Kodaira type
$$
III^* + III^* +6 I_1.
$$
on $S_a$.
The $E_6$-singularity of (\ref{SingularFibre}) turns into  an $E_7$-singularity.
\end{itemize}

\end{lem}

\subsection{Local period mapping}

Set
\begin{align}\label{A0}
\tilde{\mathfrak{A}} = \mathfrak{A}-\{a\in \mathbb{C}_a \mid a_0 a_{14} d_{84}(a)=0\}.
\end{align}
Let $F$ be a general fibre for the elliptic surface $\pi_a$ and $O$ be the zero section.
Let $C_0,\ldots,C_7$ ($D_0,\ldots,D_6$, resp.) be nodal curves in the singular fibre of type $III^*$ ($IV^*$, resp.) indicated Figure 1.
For $a\in \tilde{\mathfrak{A}}$,
the lattice generated by 
\begin{align}\label{Divisors}
F, O, C_1,\ldots , C_7, D_1, \ldots, D_6
\end{align}
is a sublattice of ${\rm NS}(S_a)$ whose intersection matrix is 
\begin{align}\label{LatticeM}
{\bf M}={\bf M}_0=U\oplus E_7(-1) \oplus E_6(-1).
\end{align}
Let $L_{K3} $ be the $K3$ lattice : $L_{K3}=II_{3,19}$.
Applying \cite{NiBilinear} Theorem 1.14.4 (see also \cite{Mo} Theorem 2.8),
we can see that ${\bf M}$ has a unique primitive embedding into $L_{K3} $ up to isometry.
The orthogonal complement of ${\bf M}$ with respect to the unimodular lattice $L_{K3}$ is given by ${\bf A}$ of (\ref{LatticeA}).
We have an isometry
$
\psi : H_2(S_a,\mathbb{Z}) \rightarrow L_{K3}
$
such that
\begin{align*}
\psi(F)=\gamma_8, \quad \psi(O)=\gamma_9, \quad \psi(C_j)=\gamma_{9+j},\quad \psi(D_k)=\gamma_{16+k} \quad (j\in\{1,\ldots,7\},k\in\{1,\ldots,6\}).
\end{align*}
Then,
the sublattice 
$\langle \gamma_8,\ldots, \gamma_{22} \rangle_\mathbb{Z}$ in $L_{K3}$ is isometric to the lattice ${\bf M}$ of (\ref{LatticeM}).
This is a primitive sublattice, because $|{\rm det}({\bf M})|=6$ is square-free.
Hence,
we can take $\gamma_1,\ldots,\gamma_7\in L_{K3}$
such that
$\{\gamma_1,\ldots,\gamma_7,\gamma_8,\ldots,\gamma_{22}\}$
gives a  basis of $L_{K3}$.
Let $\{\delta_1,\ldots,\delta_{22}\} $ be the dual basis of $\{\gamma_1,\ldots,\gamma_{22}\}$ with respect to the unimodular lattice $L_{K3}.$
Then,
the intersection matrix of the sublattice $\langle \delta_1, \ldots, \delta_7\rangle_\mathbb{Z}$ is equal to the intersection matrix of ${\bf A}$.

\begin{prop}(The canonical form for $a_0\not=0$)
If $a\in \mathfrak{A} \cap \{a_0\not =0\}$,
(\ref{SK3}) is transformed to the Weierstrass equation
\begin{align}\label{SK3Can}
S(u): z^2 = y^3 +( x^5 + u_4 x^4 w^4+ u_8 x^3 w^8) y +(u_2 x^7 w^2 + u_6 x^6 w^6 + u_{10} x^5 w^{10} + u_{14} x^4 w^{14}).
\end{align}
\end{prop}

\begin{proof}
By putting
$$
x\mapsto \frac{x}{a_0},\quad y\mapsto \frac{y}{a_0^2},\quad z\mapsto\frac{z}{a_0 ^3},
$$
and
$$
u_2=\frac{a_2}{a_0},\quad u_4=a_4,\quad u_6=a_6,\quad u_8=a_0a_8,\quad u_{10}=a_0a_{10},\quad u_{14}=a_0^2a_{14},
$$
we obtain (\ref{SK3Can}).
\end{proof}

Here, we put $u=(u_k)=(u_2,u_4,u_6,u_8,u_{10},u_{14})\in \mathbb{C}^6-\{0\}=: \mathbb{C}_u.$
Let $\mathcal{U}^*$ be a subset of $\mathbb{C}_u$ of codimension $3$ such that
\begin{align}\label{UHartogus}
\mathbb{C}_u-\mathcal{U}^* = \{u \in \mathbb{C}_u \mid u_{10}=u_{12}=u_{18}=0\}.
\end{align}
 For $\lambda \in \mathbb{C}^* $ and $u=(u_k)\in \mathcal{U}$, set $\lambda \cdot u=(\lambda^k u_k)$.
This action  induces an isomorphism $\lambda: S(u) \rightarrow S(\lambda\cdot u)$.
Letting $[u]=(u_2:u_4:u_6:u_8:u_{10}:u_{14}) \in \mathbb{P}(2,4,6,8,10,14)$ be the point which is corresponding  to $u\in \mathbb{C}_u$,
we set
$
\mathcal{U}=\{[u] \in \mathbb{P}(2,4,6,8,10,14)  \mid u \in \mathcal{U}^* \}.
$
The above action of $\mathbb{C}^*$ on $\mathcal{U}^*$, we naturally defines the family
\begin{align}\label{SUFamily}
\{S([u])\mid [u]\in \mathcal{U} \} \rightarrow \mathcal{U}.
\end{align}

\begin{df}
Let $\pi_1:S_1 \rightarrow \mathbb{P}^1(\mathbb{C})$ and $\pi_2:S_2 \rightarrow \mathbb{P}^1(\mathbb{C})$ be two elliptic surfaces.
Suppose that there exist a biholomorphic mapping $f:S_1\rightarrow S_2$ and $\varphi \in {\rm Aut}(\mathbb{P}^1(\mathbb{C}))$ with
$\varphi \circ \pi_1 = \pi_2 \circ f$.
Then, these two elliptic surfaces are said to be isomorphic as elliptic surfaces.
\end{df}

The canonical form (\ref{SK3Can}) naturally gives an elliptic surface $\pi_{[u]}:S([u])\rightarrow \mathbb{P}^1(\mathbb{C})$.

\begin{lem}\label{LemEllipticSurface}
Two elliptic surfaces $\pi_{[u_1]}: S([u_1]) \rightarrow \mathbb{P}^1(\mathbb{C})$ and $\pi_{[u_2]}:S([u_2]) \rightarrow \mathbb{P}^1(\mathbb{C})$ are isomorphic as elliptic surfaces
if and only if $[u_1]=[u_2] \in \mathbb{P}(2,4,6,8,10,14)$.
\end{lem}

\begin{proof}
We can prove it by an argument which is similar to the proof of \cite{Na} Lemma 1.1.
\end{proof}

Let us take a generic point $a\in \tilde{\mathfrak{A}}$ of (\ref{A0}).
Since $\tilde{\mathfrak{A}} \subset \mathfrak{A} \cap \{a_0\not=0\}$,
we obtain the corresponding surface $S([u])$  for a  parameter $[u]\in\mathbb{P}(2,4,6,8,10,14)={\rm Proj}(\mathbb{C}[u_2,u_4,u_6,u_8,u_{10},u_{14}])$, which is given by the canonical form (\ref{SK3Can}).
We set 
$$
\mathcal{P}\left(\tilde{\mathfrak{A}}\right)=\left\{[u] \in \mathbb{P}(2,4,6,8,10,14) \mid \text{there exists }  a\in \tilde{\mathfrak{A}}  \text{ such that } S_a \text{ is identified with  } S([u])\text{ of } (\ref{SK3Can})\right\}.
$$
By an argument which is similar to \cite{Na} p.41, using Lemma \ref{LemEllipticSurface} also,
we can obtain a local period mapping defined on sufficiently small neighborhood around $[u]$ in $\mathcal{P}\left(\tilde{\mathfrak{A}}\right)$. 
By gluing the local period mappings, we obtain the period mapping 
\begin{align}\label{Phi1}
\Phi_1: \mathcal{P}\left(\tilde{\mathfrak{A}}\right) \rightarrow \mathcal{D}
\end{align}
given by
$$
[u] \mapsto \left( \int_{\psi^{-1}_{[u]}(\gamma_1)} \omega_{[u]}: \cdots :  \int_{\psi^{-1}_{[u]}(\gamma_7)} \omega_{[u]} \right),
$$
where $\omega_{[u]}$ is a unique holomorphic $2$-form on $S([u])$ up to a constant factor
and
$$
\psi_{[u]}: H_2(S([u]),\mathbb{Z}) \rightarrow L_{K3}
$$
is an appropriate isometry, called $S$-marking.
We note that the period mapping (\ref{Phi1}) is a multivalued analytic mapping.

We call the pair $(S([u]), \psi_{[u]})$ an $S$-marked $K3$ surface.
By applying  Torelli's theorem to the above local period mappings, 
 we can show the following theorem as in the proof of \cite{Na} Theorem 1.1 and Corollary 1.1.

\begin{thm}\label{TheoremLattice}
For a generic point $a\in \tilde{\mathfrak{A}}$ of (\ref{A0}),
the Picard number of $S_a$ is $15$.
The intersection matrix of the N\'eron-Severi lattice ${\rm NS}(S_a)$ (the transcendental lattice ${\rm Tr}(S_a)$, resp.)
 is equal to ${\bf M}$ of (\ref{LatticeM})  (${\bf A}$ of (\ref{LatticeA}), resp.).
\end{thm}

\subsection{Double covering $K_a$ of $S_a$}

The $K3$ surface $S_a$ of (\ref{SK3}) is transformed to
\begin{align}\label{ShigaForm}
 Z^2 = Y^3+ \Big(a_4 + a_{0} X+\frac{a_8}{X}\Big)Y+\Big(a_6 + a_2 X + \frac{a_{10}}{X} + \frac{a_{14}}{X^2}\Big)
\end{align}
by the birational transformation
$$
x\mapsto X, \quad y\mapsto X^2 Y, \quad z\mapsto X^3 Z.
$$
We have a double covering
\begin{align}\label{KL}
K_a: Z^2 = Y^3+ \Big(a_4 + a_{0} U^2 +\frac{a_8}{U^2}\Big)Y+\Big(a_6 + a_2 U^2 + \frac{a_{10}}{U^2} + \frac{a_{14}}{U^4}\Big)
\end{align}
of (\ref{ShigaForm}).
There is a Nikulin involution on $K_a$ given by 
$$
\iota_{K_a}: (U,Y,Z)\mapsto (-U,Y,-Z).
$$
This means that it satisfies $\iota_{K_a}^*\omega_K=\omega_K$, where $\omega_K$ is a unique holomorphic $2$-form up to a constant factor.

We have a family $\overline{\varpi_0}: \mathfrak{G}_0\rightarrow \mathfrak{A}_0$ of $K3$ surfaces, where $\mathfrak{G}_0=\{K_a\text{ of } (\ref{KL}) \mid a\in \mathfrak{A}_0\}$.
In fact, this surface $K_a$ can be regarded as a natural generalization of the Kummer surface for a principally polarized Abelian surface. 
Therefore, 
we call a member of $\mathfrak{G}_0$ a Kummer-like surface in this paper.
More precisely, see Section 4.
We remark that the double covering $\varphi_0$ in (\ref{TriDiagram}) is given by this $\iota_{K_a}$.

\begin{rem}\label{RemSekiguchi}
The surface $K_a$ of (\ref{KL}) has another involution
$$
\jmath_{K_a}: (U,Y,Z)\mapsto (-U,Y,Z).
$$
This is not a Nikulin involution.
The minimal resolution of the quotient surface $K/\langle \jmath_{K_a} \rangle$ is given by the equation
$$
\Sigma: z'^2=y'^3+
( a_0 x'^3 +  a_4 x'^2 +  a_8 x') y'
+ (a_2 x'^4+ a_6 x'^3+ a_{10} x'^2 + a_{14} x').
$$
This is a rational surface.
The surface $\Sigma$ is very similar to a surface appearing in Sekiguchi's recent work \cite{S},
in which he studies
algebraic Frobenius potentials, deformation of singularities and  Arnold's problem.
Namely,
he obtains the equation in the form
$$
z^2=f_{E_7(1)}:=y^3+(x^3+t_2 x^2+t_4 x)y + (t_1 x^4 + t_3 x^3 + t_5 x^2 + t_7 x)+s_3 y^2
$$
of a rational surfaces.
Putting $s_3=0$, we can see an apparent correspondence to our $\Sigma.$ 

The author is expecting that there is a non-trivial and unrevealed theory
connecting our work of the moduli of $K3$ surfaces
and Sekiguchi's result of Frobenius potentials.
\end{rem}

\section{Moduli space of ${\bf M}$-polarized $K3$ surfaces}

In this section,
letting ${\bf M}$ be the lattice of (\ref{LatticeM}) of signature $(1,14)$,
we consider the moduli space of ${\bf M}$-polarized $K3$ surfaces. 
An ${\bf M}$-polarized $K3$ surface is a pair $(S,j)$ of a $K3$ surface $S$ and a primitive embedding $j:{\bf M} \hookrightarrow {\rm NS}(S)$.
Two {\bf M}-polarized $K3$ surfaces $(S_1,j_1)$ and $(S_2,j_2) $ are said to be isomorphic if there exists an isomorphism $f:S_1\rightarrow S_2$ of $K3$ surfaces such that $j_2=f_*\circ j_1.$
In this paper, ${\rm NS}(S)$ is often regarded as a sublattice of the homology group $H_2(S,\mathbb{Z})$. 
We note that  ${\rm NS}(S)$  is identified with the sublattice $H^2(S,\mathbb{Z}) \cap H^{1,1}(S,\mathbb{R})$ of the cohomology group $H^2(S,\mathbb{Z})$ by the Poincar\'e duality.
This is denoted  by the same notation in the discussion below.

Let $V(S)^+$ be the connected component of $V(S)=\left\{x\in H^{1,1}_\mathbb{R} (S) \mid (x\cdot x)>0 \right\}$
which contains the class of a  K\"ahler form on $S$.
Let $\Delta(S)^+$ be the subset of effective classes of $\Delta(S)=\{\delta\in {\rm NS}(S) \mid (\delta\cdot \delta)=-2\}.$
Set $C(S)=\{x\in V(S)^+ \mid (x,\delta)\geq 0, \text{ for all } \delta \in \Delta(S)^+\}$.
The subset $C(S)^+$ of $C(S)$, which is defined by the condition $(x\cdot\delta)>0$, is called the K\"ahler cone.
We set
${\rm NS}(S)^+=C(S) \cap H^2(S,\mathbb{Z})$ and ${\rm NS}(S)^{++}=C(S)^+ \cap H^2(S,\mathbb{Z}).$

Due to Theorem \ref{TheoremLattice},
we can take a point $\bar{a}\in \tilde{\mathfrak{A}}$ and an $S$-marking $\psi_0: H_2(S_0,\mathbb{Z})\rightarrow L_{K3}$
such that  $\psi_0^{-1}({\bf M}) = {\rm NS}(S_0)$, where $S_0=S_{\bar{a}}$ is called a reference surface.
Letting $\Delta({\bf M})=\{\delta \in {\bf M} \mid (\delta\cdot\delta)=-2\}$,
we set
$\Delta({\bf M})^+ = \left\{\delta\in \Delta({\bf M}) \mid \psi_0^{-1}(\delta) \in {\rm NS}(S_0) \text{ gives an effective class} \right\}$.
The set $V({\bf M})=\{y\in {\bf M}_\mathbb{R} \mid (y\cdot y)>0\}$ has two connected components.
We suppose the component $V({\bf M})^+$ contains $\psi_0(x)$ for $x\in V(S_0)^+$.
Set $C({\bf M})^+=\left\{y\in V({\bf M})^+ \mid (y\cdot\delta)>0, \text{ for all } \delta\in \Delta({\bf M})^+ \right\}$.
An ${\bf M}$-polarized $K3$ surface $(S,j)$ is called a pseudo-ample ${\bf M}$-polarized $K3$ surface if $j(C({\bf M})^+) \cap {\rm NS}(S)^+ \not= \phi.$

For a $K3 $ surface $S$,
let $\psi:H_2(S,\mathbb{Z}) \rightarrow L_{K3}$ be an isometry of lattices with $\psi^{-1}({\bf M}) \subset {\rm NS}(S)$.
We call the pair $(S,\psi)$ of such $S$ and $\psi$ is called a marked $K3$ surface.
If $(S,\psi^{-1}|_{\bf M}) $ is a pseudo-ample  ${\bf M}$-polarized $K3$ surface,
then $(S,\psi)$ is called a pseudo-ample marked ${\bf M}$-polarized $K3$ surface.
For two pseudo-ample marked ${\bf M}$-polarized $K3$ surfaces  $(S_1,\psi_1)$ and $(S_2,\psi_2)$,
we suppose $(S_1,\psi_1^{-1}|_{\bf M})$ and $(S_2,\psi_2^{-1}|_{\bf M})$ are isomorphic as ${\bf M}$-polarized $K3$ surfaces.
Then,  $(S_1,\psi_1)$ and $(S_2,\psi_2)$ are said to be isomorphic as pseudo-ample ${\bf M}$-polarized $K3$ surfaces.
Also, if there is an isomorphism $f: S_1 \rightarrow S_2$ such that $\psi_1 = \psi_2 \circ f_*$,
we say $(S_1,\psi_1)$ and $(S_2,\psi_2)$ are isomorphic as pseudo-ample marked ${\bf M}$-polarized $K3$ surfaces.

By gluing local moduli spaces of marked ${\bf M}$-polarized $K3$ surfaces,
we have the fine moduli space $\mathcal{M}_{\bf M}$ of marked ${\bf M}$-polarized $K3$ surfaces. 
Then, we have the period mapping
\begin{align}\label{perM}
{\rm per}: \mathcal{M}_{\bf M} \rightarrow \mathcal{D}_{\bf M},
\end{align}  
where $\mathcal{D}_{\bf M}$ is given in (\ref{DM}).

Let $\mathcal{M}_{\bf M}^{\rm pa} (\subset \mathcal{M}_{\bf M})$ be the set of isomorphism classes of pseudo-ample marked {\bf M}-polarized $K3$ surfaces.
By restricting (\ref{perM}) to $\mathcal{M}_{\bf M}^{\rm pa}$,
we have  a surjective mapping
\begin{align}\label{perP}
{\rm per}' : \mathcal{M}_{\bf M}^{\rm pa} \rightarrow \mathcal{D}_{\bf M}.
\end{align}
The group $\Gamma ({\bf M})=\{ \sigma\in O(L_{K3}) \mid \sigma(m)=m \text{ for all } m\in {\bf M}\}$
acts on $\mathcal{M}_{\bf M}$ by $(S,\psi) \mapsto (S,\psi \circ \sigma).$
Then, $\mathcal{M}_{\bf M}^{\rm pa}/ \Gamma({\bf M})$ gives the set of isomorphism classes of pseudo-ample ${\bf M}$-polarized $K3$ surfaces.

\begin{thm} (Dolgachev \cite{D}, Section 3) \label{ThmDolgachev}
The period mapping (\ref{perP}) induces the bijection
$$
\mathcal{M}_{\bf M}^{\rm pa} /\Gamma({\bf M}) \simeq \mathcal{D}_{\bf M}/\tilde{O}({\bf A})= \mathcal{D}/\Gamma,
$$
where $\Gamma$ is given in (\ref{Gamma}).
Especially, $\mathcal{D}/\Gamma$ gives the set of isomorphism classes of pseudo-ample ${\bf M}$-polarized $K3$ surfaces.
\end{thm}

Let us take a reference surface $S_0=S_{\bar{a}}$ for $\bar{a}\in\tilde{\mathfrak{A}}$
with the divisors (\ref{Divisors})
 and the $S$-marking $\psi_0: H_2(S_0,\mathbb{Z})\rightarrow L_{K3}$ such that ${\rm NS}(S_0)=\psi_{0}^{-1} ({\bf M}).$
For a pseudo-ample marked ${\bf M}$-polarized $K3$ surface $(S,\psi)$,
as in the proof of \cite{Na} Theorem 2.3,
we can show that there is an isometry $\psi:H_2(S,\mathbb{Z}) \rightarrow L_{K3}$ satisfying the following conditions:
\begin{itemize}

\item[(i)] $\psi^{-1} ({\bf M}) \subset {\rm NS}(S)$,

\item[(ii)] $\psi^{-1} \circ \psi_0(F)$,$\psi^{-1}\circ \psi_0(O)$,$\psi^{-1}\circ \psi_0(C_j)$ $(j\in \{1,\ldots, 7\})$, $\psi^{-1}\circ \psi_0(D_k)$ $(k\in \{1,\ldots, 6\})$ are effective divisors,  

\item[(iii)] $\psi^{-1}\circ \psi_0(F)$ is a nef divisor.

\end{itemize}
Such an isometry $\psi$ is called the $P$-marking in \cite{Na}.

By an argument which is similar to the proof of \cite{Na} Lemma 2.1, we can prove the following lemma.

\begin{lem}\label{LemMarking}
For any  pseudo-ample marked ${\bf M}$-polarized $K3$ surface $(S,\psi)$, 
 there exists $a\in \mathfrak{A}$ such that
 $(S,\psi)$ is given by the elliptic $K3$ surface $\pi_a:S_a  \rightarrow \mathbb{P}^1(\mathbb{C})$ given by the Weierstrass equation (\ref{SK3}).
 Especially,  the divisor $\psi^{-1} \circ \psi_0 (F)$, which is effective and nef, gives  a general fibre for $\pi_a$.
\end{lem}

Let us take two pseudo-ample marked ${\bf M}$-polarized $K3$ surfaces given by elliptic $K3$ surfaces  $\pi_a$ and $\pi_{a'}$ in the sense of Lemma \ref{LemMarking}.
If they are isomorphic, 
then the types of the  singular fibres for $\pi_a$ coincide with those for $\pi_{a'}$.
According to Lemma \ref{LemSingular},  we have the following facts:
\begin{itemize}

\item for $a\in \{a_0\not= 0\}$, then $\pi_a^{-1}(\infty)$ is of Kodaira type $III^*$,

\item for $a\in \{a_0= 0\}$, then $\pi_a^{-1}(\infty)$ is of Kodaira type $II^*$.

\end{itemize}
Hence,  $S_a$ for $a\in \{a_0\not=0\}$ is not isomorphic to $S_{a'}$ for $a'\in \{a_0\not=0\}$ as pseudo-ample ${\bf M}$-polarized $K3$ surfaces. 
If $a_0\not=0$, we have a canonical form given by (\ref{SK3Can}). 
On the other hand, if $a_0=0$, we have the following result.

\begin{prop}(The canonical form for $a_0 =0$)
If $a\in \mathfrak{A} \cap \{a_0 =0\}$,
(\ref{SK3}) is transformed to the Weierstrass equation
\begin{align}\label{SK3CanK}
S_1(t): z^2 = y^3 +(  t_4 x^4 w^4+ t_{10} x^3 w^{10}) y +( x^7  + t_6 x^6 w^6 + t_{12} x^5 w^{12} + t_{18} x^4 w^{18}).
\end{align}
\end{prop}

\begin{proof}
By putting
$$
x\mapsto \frac{x}{a_2 },\quad y\mapsto \frac{y}{a_2^2 },\quad z\mapsto\frac{z}{a_2 ^3 },
$$
and
$$
t_4=a_4,\quad t_6=a_6,\quad t_{10}=a_2 a_8,\quad t_{12}=a_2 a_{10},\quad t_{18}=a_2^2 a_{14},
$$
we obtain (\ref{SK3CanK}).
\end{proof}

Now, we naturally obtain $S_1([t])$ for $[t]\in \mathbb{P}(4,6,10,12,18)={\rm Proj}(\mathbb{C}[t_4,t_6,t_{10},t_{12},t_{18}])$.
Let $\mathcal{T}$ be a subset of  $\mathbb{P}(4,6,10,12,18)$  such that
$\mathbb{P}(4,6,10,12,18)-\mathcal{T}=\{t_{10}=t_{12}=t_{18}=0\}$.

The above argument guarantees that
the set of isomorphism classes of pseudo-ample ${\bf M}$-polarized $K3$ surfaces is given by the disjoint union $\mathcal{U} \sqcup \mathcal{T}.$
Therefore, we have the following result.

\begin{thm}\label{ThmPeriodP}
The period mapping in Theorem \ref{ThmDolgachev} has an explicit form
\begin{align}\label{UTPeriod}
\Phi: \mathcal{U} \sqcup \mathcal{T} \simeq \mathcal{D}/\Gamma,
\end{align}
where $\Gamma$ is given in (\ref{Gamma}).
Especially, 
the injection $ \mathcal{P}\left(\tilde{\mathfrak{A}}\right) \hookrightarrow \mathcal{D}/\Gamma$,
which is induced by the visualized period mapping $\Phi_1$ of (\ref{Phi1}),
is extended to (\ref{UTPeriod}).
\end{thm}

Let $\mathcal{H}$ be the arithmetic arrangement of hyperplanes defined in Section 1.1.
By virtue of Lemma \ref{LemSingular},
$\mathcal{H} $ corresponds to pseudo-ample ${\bf M}$-polarized $K3$ surfaces given by the canonical form (\ref{SK3CanK}).
So, from Theorem \ref{ThmPeriodP},
the restriction of the period mapping $\Phi$ of (\ref{UTPeriod}) gives an isomorphism 
\begin{align}\label{PhiTIso}
\Phi |_\mathcal{T}:\mathcal{T} \simeq \left(\bigcup_{H\in\mathcal{H}} \mathcal{D}_H\right)/\Gamma.
\end{align}
We remark that (\ref{PhiTIso}) coincides with the period mapping of \cite{Na} Corollary 2.1.
The detailed results of (\ref{PhiTIso}) will be summarized in Section 4.

For $\mathcal{D}^\circ$ of (\ref{Dcirc}), we have an isomorphism
\begin{align}\label{PhiUIso}
\Phi |_\mathcal{U} : \mathcal{U} \simeq \mathcal{D}^\circ/\Gamma.
\end{align}
Since $\mathcal{P}\left(\tilde{\mathfrak{A}}\right)\subset \mathcal{U}$, 
(\ref{PhiUIso}) is an extension of $\mathcal{P}\left(\tilde{\mathfrak{A}}\right)\hookrightarrow \mathcal{D}/\Gamma$,
which is derived from $\Phi_1$ of (\ref{Phi1}).
By abuse of notation,
this $\Phi|_\mathcal{U}$ will be denoted by $\Phi$
in Section 5.

\section{Sequence of families of $K3$ surfaces and complex reflection groups}

Let  $\mathfrak{A}_1,\mathfrak{A}_2,\mathfrak{A}_3$ be subvarieties of $\mathfrak{A}_0$ explicitly given by
\begin{align}\label{ParameterSubA}
\begin{cases}
&\mathfrak{A}_1=\{a\in \mathfrak{A}_0  \mid a_0=0\},\\
& \mathfrak{A}_2=\{a\in \mathfrak{A}_0   \mid a_0=a_{14}=0\},\\
&\mathfrak{A}_3=\{a\in \mathfrak{A}_0   \mid a_0=a_{14}=\mathfrak{M}(a)=0\}.
\end{cases}
\end{align} 
Here,
$$
\mathfrak{M}(a):=\left(a_2 a_{10}  + \frac{a_4^3}{27} - \frac{a_6^2}{4}\right)^2 + \frac{1}{27} a_4 (a_4 a_6 + 6 a_2 a_8)^2=0
$$
coincides with the modular equation of the Humbert surface for the minimal discriminant which is studied in \cite{NSA} Theorem 5.4 via an appropriate transformation.

We have the subfamilies
$$
\varpi_j: \mathfrak{F}_j \rightarrow \mathfrak{A}_j \quad \quad (j\in \{1,2,3\})
$$
of $\varpi_0: \mathfrak{F}_0=\{S_a \mid a\in \mathfrak{A}_0\} \rightarrow \mathfrak{A}_0$.
They are indicated in the diagram (\ref{diagram1}).
Also, we need another subvariety
$
\mathfrak{A}_1'=\{a\in \mathfrak{A}_0  \mid a_{14}=0\}
$
and another subfamily $\varpi'_1:\mathfrak{F}_1' \rightarrow \mathfrak{A}_1'$  such that $\mathfrak{A}_1\cap \mathfrak{A}_1'=\mathfrak{A}_2$.

\subsection{Transcendental lattices for subfamilies of $\mathfrak{F}_0$}

The above mentioned subfamilies of $\mathfrak{F}_0$
are precisely studied in \cite{CD}, \cite{NHilb}, \cite{CMS} and \cite{Na}.
For each case,
it is important to determine the transcendental lattice in order to study the moduli space of the corresponding lattice polarized $K3$ surfaces.
By surveying the results of those papers, we have the following result.

\begin{prop}\label{PropLatticeF}
The intersection matrices of the transcendental lattices of a generic member of the subfamilies $\mathfrak{F}_1,\mathfrak{F}_1',\mathfrak{F}_2,\mathfrak{F}_3$ of $\mathfrak{F}_0$ are given in Table 2.
\begin{table}[h]
\center
\vspace{-3mm}
\caption{Subfamilies of $\mathfrak{F}_0$}
\vspace{1mm}
\begin{tabular}{cccccc}
\toprule
Parameters  &  Transcendental lattices   & $K3$ surfaces \\
 \midrule
$\mathfrak{A}_1$ & ${\bf A}_1=U\oplus U\oplus A_2(-1)$  & \cite{Na} \\
$\mathfrak{A}_1'$ & ${\bf A}_1'=U\oplus U\oplus A_1(-1)^{\oplus 2}$  &  \cite{CMS} \\
$\mathfrak{A}_2$  &  ${\bf A}_2=U\oplus U \oplus A_1(-1)$  & \cite{CD}  \\
$\mathfrak{A}_3$  &  ${\bf A}_3=U\oplus \begin{pmatrix}  2 & 1 \\ 1 & -2\end{pmatrix}$  & \cite{NHilb}\\
 \bottomrule
\end{tabular}
\end{table} 
\end{prop}

The lattices in (\ref{LatticesAj}) are based on
 Proposition \ref{PropLatticeF} and Theorem \ref{TheoremLattice}.

\subsection{Transcendental lattices for subfamilies of $\mathfrak{G}_0$ of Kummer-like surfaces}

In Section 2.3,
we have the family $\overline{\varpi_0}: \mathfrak{G}_0=\{K_a \mid a\in \mathfrak{A}_0\}\rightarrow \mathfrak{A}_0$ of $K3$ surfaces.
This family contains interesting and important families of algebraic $K3$ surfaces.
For example,
the subfamily $\mathfrak{G}_2$ over $\mathfrak{A}_2$ coincides with the well-known family of Kummer surfaces derived from principally polarized Abelian surfaces.
Also, the subfamily $\mathfrak{G}_1'$ over $\mathfrak{A}_1'$ is the family of $K3 $ surfaces
given by the double covering of $\mathbb{P}^2(\mathbb{C})$ branched along six lines, studied by Matsumoto-Sasaki-Yoshida \cite{MSY}.

\begin{prop}\label{PropLatticeG}
The intersection matrices of the transcendental lattices of a generic member of the subfamilies $\mathfrak{G}_1,\mathfrak{G}_1',\mathfrak{G}_2,\mathfrak{G}_3$ of $\mathfrak{G}_0$ are given in Table 3.
\begin{table}[h]
\center
\vspace{-3mm}
\caption{Subfamilies of $\mathfrak{G}_0$ }
\vspace{1mm}
\begin{tabular}{cccccc}
\toprule
Parameters  &  Transcendental lattices   & $K3$ surfaces \\
 \midrule
$\mathfrak{A}_1$ & ${\bf B}_1=U(2)\oplus U(2)\oplus A_2(-2)$  & \cite{NS2} \\
$\mathfrak{A}_1'$ & ${\bf B}_1'=U(2)\oplus U(2)\oplus A_1(-1)^{\oplus 2}$  &  \cite{MSY} \\
$\mathfrak{A}_2$  &  ${\bf B}_2=U(2)\oplus U(2) \oplus A_1(-2)$  &  Kummer surface  \\
$\mathfrak{A}_3$  &  ${\bf B}_3=U(2)\oplus \begin{pmatrix}  4 & 2 \\ 2 & -4\end{pmatrix}$  & \cite{NaKum}\\
 \bottomrule
\end{tabular}
\end{table} 
\end{prop}

By the way, 
we are able to  determine  the transcendental lattices  in Table 2
for the subfamilies $\mathfrak{F}_0$ by a relatively simple way.
However,
the proof of Proposition \ref{PropLatticeG}
is much more complicated.
We need a delicate argument for each case.
For detail, see the beginning of Section 6.

Proposition \ref{PropLatticeG} is necessary for the proof of 
Theorem \ref{ThmTrK}, which is the main theorem of Section 6.
The lattices in (\ref{LatticesBj}) are based on Proposition \ref{PropLatticeG} and Theorem \ref{ThmTrK}.

\subsection{Relation between modular forms and invariants of complex reflection groups via theta functions}

For the subfamilies $\mathfrak{F}_1,\mathfrak{F}_2$ and $\mathfrak{F}_3$,
there is a non-trivial relationship between the period mappings for them 
and a complex reflection group of rank $r_j=6-j$.

In each family of $K3$ surfaces in Table 1,
we have the period mapping
\begin{align}\label{Phij}
\Phi_j: \mathcal{P}_j \simeq \mathcal{D}_j/\Gamma_j
\end{align}
where $\mathcal{P}_j$ is a Zariski open set in the weighted projective space
whose weights are given in Table 1, 
$\mathcal{D}_j$ is a $(5-j)$-dimensional symmetric domain
and $\Gamma_j$ is a subgroup of the orthogonal group of the lattice ${\bf A}_j$ in Table 2.
We note that $\Phi_j$ of  (\ref{Phij}) can be obtained as a restriction of (\ref{UTPeriod}).

\begin{table}[h]
\center
\vspace{-3mm}
\caption{$K3$ surfaces, reflection groups and theta functions}
\vspace{1mm}
\begin{tabular}{ccccccc}
\toprule
  $j$ &      Reflection groups  &  Theta functions  &Theta expressions via $K3$ surfaces  \\
 \midrule
$1$   &   No.33 &\cite{DK} &\cite{NS1}  \\
$2$  &     No.31  &\cite{R} & \cite{CD} and \cite{R} \\
$3$ &  No.23  & \cite{Muller} & \cite{NHilb}  \\
 \bottomrule
\end{tabular}
\end{table}

Finite complex reflection groups are listed by Shephard-Todd \cite{ST} (see also \cite{LT}).
Note that  real reflection groups are contained in this list.

A complex reflection group of rank $r_j$ acts on the polynomial ring $\mathbb{C}[X_1,\ldots,X_{r_j}]$.
We can find generators of the ring of invariants for this action.
Let $\left( w_1^{(j)},\ldots,w_{r_j}^{(j)} \right)$  $(j\in\{1,2,3\})$ be the set of weights given in Table 1.
There is a system $\left\{g_{\kappa_j w_1}^{(j)}(X_1,\ldots,X_{r_j}),\ldots, g_{\kappa_j w_{r_j}}^{(j)}(X_1,\ldots,X_{r_j}) \right\}$ of generators of the ring of invariants.
Here, $\kappa_j \in \mathbb{Z}$ is the integer in Table 1
and $g^{(j)}_{\kappa_j w_l}$ is a polynomial of degree $\kappa_j w_l^{(j)}$.
For example,
if $j=3$,
the invariants for the group No.23 are famous  Klein's icosahedral invariants introduced  in \cite{K}.
In the preceding papers in Table 4, we have the  explicit theta expressions of the inverse of $\Phi_j$ of (\ref{Phij}) via appropriate systems of  theta functions.

For $j=1,2$,
 we have simple expressions of the above mentioned results.
There exists a system $\{\vartheta_1^{(j)} (Z_j), \ldots,\vartheta_{r_j}^{(j)} (Z_j)\}$ 
of theta functions  $\mathcal{D}_j \ni Z_j \mapsto \vartheta_\ell^{(k)} (Z_j)\in \mathbb{C}$ of weight $1/\kappa_j$
such that
\begin{align}\label{ModularTheta}
\mathcal{D}_j \ni Z_j \mapsto \left(g_{\kappa_j w_1}^{(j)}\left(\vartheta_1^{(j)} (Z_j), \ldots,\vartheta_{r_j}^{(j)} (Z_j)\right): \ldots : g_{\kappa_j w_{r_j}}^{(j)}\left(\vartheta_1^{(j)} (Z_j), \ldots,\vartheta_{r_j}^{(j)} (Z_j)\right)\right)\in \mathcal{P}_j
\end{align}
gives a ratio  of modular forms on $\mathcal{D}_j$ with respect to $\Gamma_j$
and
 it coincides with the inverse of the period mapping $\Phi_j $ of (\ref{Phij}).

\begin{itemize}

\item If $j=1$,
the invariants for the group No.33 are given in \cite{B}
and the explicit form (\ref{ModularTheta}) is established in \cite{NS1} Theorem 5.1,
using the theta functions of \cite{DK}. 
In this case,  (\ref{ModularTheta}) is given by a ratio of Hermitian modular forms for the unitary group $U(2,2)$
concerned with the imaginary quadratic field of the simplest discriminant.

\item If $j=2$,
using the invariants for the group No.31 and the theta functions given in \cite{R} Section 4,
one can obtain the expression (\ref{ModularTheta}) by combining the results  \cite{CD} Theorem 3.5 and \cite{R} Section 4 (see also, \cite{NS1} Section 5.1).
In this case,  (\ref{ModularTheta}) is given by a ratio of well-known Siegel modular forms of degree $2$.

\end{itemize}
Also, refer to Remark \ref{RemKappa}.

\section{Meromorphic modular forms}

Let $\mathcal{D}^*$
be the connected component of 
$
\left\{\xi\in {\bf A}\otimes \mathbb{C} \mid (\xi\cdot\xi)=0, (\xi\cdot\overline{\xi})>0 \right\}
$
which projects to $\mathcal{D}$.
For $\mathcal{D}^\circ$ of (\ref{Dcirc}),
let $(\mathcal{D}^\circ)^*$  be a subset of $\mathcal{D}^*$ which projects to $\mathcal{D}^\circ$.

Based on the fact stated in Theorem \ref{ThmLooijenga} below,
we will use the following terminology.

\begin{df}\label{DfMeroModularForm}
A holomorphic function $f: (\mathcal{D}^\circ)^* \rightarrow \mathbb{C}$ given by $Z \mapsto f(Z)$
is called a meromorphic modular form  of weight $k\in \mathbb{Z}$ and character $\chi \in {\rm Char}(\Gamma)$ with poles in $\mathcal{H}$,
if $f$ satisfies
\begin{itemize}

\item[(i)] $f(\lambda Z) = \lambda^{-k} f(Z) \quad (\text{for all } \lambda \in \mathbb{C}^*),$

\item[(ii)] $f(\gamma Z) = \chi (\gamma) f(Z) \quad (\text{for all } \gamma \in \Gamma).$

\end{itemize}

\end{df}

The vector space of the meromorphic modular forms of weight $k\in \mathbb{Z}$ and $\chi \in {\rm Char}(\Gamma)$ with poles in $\mathcal{H}$ is denoted by $\mathcal{A}^\circ _k (\Gamma,\chi)$.
Then, the ring of the meromorphic modular forms is given by
$$
\mathcal{A}^\circ (\Gamma)=\bigoplus_{k\in \mathbb{Z}} \bigoplus_{\chi \in {\rm Char}(\Gamma)} 
\mathcal{A}^\circ _k (\Gamma,\chi).
$$
In this section, we will construct the generator of this ring.
Recalling Proposition \ref{PropKneser},
we will consider the cases of $\chi={\rm id}$ and $\chi={\rm det}$.
The structure of the ring $\mathcal{A}^\circ (\Gamma)$ is determined by Theorem \ref{ThmId} and \ref{Thmdet}.

By the way,
in \cite{HU} (see also \cite{Na}),
period mappings for lattice polarized $K3$ surfaces 
and the canonical orbibundles on the Satake-Baily-Borel compactifications of  symmetric spaces
are effective to construct holomorphic modular forms.
However, 
the Satake-Baily-Borel compactifications are useless for our purpose,
because we want to obtain not holomorphic modular forms but meromorphic modular forms.
Accordingly,
we will use the Looijenga compactification for $\mathcal{D}^\circ$ of (\ref{Dcirc}) and $\Gamma$ of (\ref{Gamma}),
instead of the Satake-Baily-Borel compactification.

Since  Proposition \ref{PropCodim} and (\ref{UHartogus}) hold, 
by Hartogs's extension theorem, 
the period mapping $\Phi$ of (\ref{PhiUIso}) is extended to the isomorphism
\begin{align}\label{HatPhi}
\widehat{\Phi}: \widehat{\mathcal{U}}\simeq \widehat{\mathcal{D}^\circ/\Gamma}^{\bf L}
\end{align}
between the weighted projective space
$
\widehat{\mathcal{U}}=\mathbb{P}(2,4,6,8,10,14)
$
 and the Looijenga compactification.
Our construction of meromorphic modular forms is based on this period mapping.

\subsection{Meromorphic modular forms of character ${\rm id}$}

There exists a unique holomorphic $2$-form $\omega_u$ on $S(u)$ of (\ref{SK3Can}) up to a constant factor.
This is explicitly given by $\frac{dx_0 \wedge dy_0 }{ z_0}, $ where $x_0=\frac{x}{w^4}, y_0=\frac{y}{w^{10}}, z_0=\frac{z}{w^{15}}.$
The action of $\lambda\in \mathbb{C}^*$ given by $S(u) \rightarrow S(\lambda \cdot u)$, 
which defines the family (\ref{SUFamily}), 
induces the relation
\begin{align}\label{omegaaction}
\lambda^* \omega_{\lambda \cdot u}= \lambda^{-1} \omega_{u}.
\end{align}

\begin{thm}\label{ThmId}
The ring $\mathcal{A}^\circ (\Gamma,{\rm id})$  of meromorphic modular forms of character ${\rm id}$ is isomorphic to the polynomial ring
$ \mathbb{C}[u_2,u_4,u_6,u_8,u_{10},u_{14}].
$ 
Here, a polynomial of weight $k$ defines a modular form of weight $k$. 
\end{thm}

\begin{proof}
We have a principal $\mathbb{C}^*$-bundle
${\rm pr} :(\mathcal{D}^\circ)^* \rightarrow \mathcal{D}^\circ$.
The quotient space $Q=\mathcal{D}^\circ/\Gamma$ is identified with a Zariski open set $\mathcal{U}$ of the weighted projective space $\widehat{\mathcal{U}}=\mathbb{P}(2,4,6,8,10,14)$
via the period mapping.
Since ${\rm pr}$ is equivalent under the action of $\Gamma=\tilde{O}^+({\bf A}),$
we have a principal $\mathbb{C}^*$-bundle
$\overline{{\rm pr}}: (\mathcal{D}^\circ)^* /\Gamma \rightarrow Q.$
Let $\mathcal{O}_Q (1) $ be the line bundle over $Q$ associated with $\overline{{\rm pr}}$ and set $\mathcal{O}_Q(k)=\mathcal{O}_Q(1)^{\otimes k}.$
Recalling the definition of the associated bundle, we can regard a section of $\mathcal{O}_Q(k) $ as a holomorphic function $(\mathcal{D}^\circ)^*\ni Z \mapsto s(Z) \in \mathbb{C}$ 
satisfying
\begin{align}\label{SectionAssBundle}
s(\lambda Z)=\lambda^{-k} s(Z), \quad s(\gamma Z)=s(Z),
\end{align}
where $\lambda\in \mathbb{C}^*$ and $\gamma\in \Gamma.$
From (\ref{UHartogus}), $\widehat{\mathcal{U}}-\mathcal{U}$ is an analytic subset such that
$
{\rm codim} \left(\widehat{\mathcal{U}}-\mathcal{U} \right)\geq 2.
$
So, via Hartogs's phenomenon,
the inclusion $\iota_\mathcal{U}:\mathcal{U} \hookrightarrow \widehat{\mathcal{U}}$ induces the isomorphism
\begin{align}\label{PicIsoU}
\iota_\mathcal{U}^*: {\rm Pic}\left(\widehat{\mathcal{U}}\right) \simeq {\rm Pic}(\mathcal{U}).
\end{align}
Now, we have ${\rm Pic}\left(\widehat{\mathcal{U}}\right) \simeq \mathbb{Z}$ and
\begin{align}\label{CuOplus}
\bigoplus_{k\in \mathbb{Z}} 
H^0\left(\widehat{\mathcal{U}}, 
\mathcal{O}_{\widehat{\mathcal{U}}} (k)\right) 
= \mathbb{C}[u_2,u_4,u_6,u_8,u_{10},u_{14}],
\end{align} 
because $\widehat{\mathcal{U}}={\rm Proj}(\mathbb{C}[u_2,u_4,u_6,u_8,u_{10},u_{14}])$ is a weighted projective space.
From (\ref{PicIsoU}) and (\ref{CuOplus}), we have
\begin{align}\label{QOplus}
\bigoplus_{\mathcal{L}\in {\rm Pic}(Q)} H^0(Q,\mathcal{O}_Q(\mathcal{L})) 
 \simeq 
  \mathbb{C}[u_2,u_4,u_6,u_8,u_{10},u_{14}].
\end{align}
Due to (\ref{omegaaction}), the period mapping gives the following diagram:
\begin{align}\label{DiagramWeight}
  \xymatrix@C=75pt{
   u=(u_2,u_4,u_6,u_8,u_{10},u_{14})  \ar[r]^{\hspace{2cm}\text{period mapping}} \ar[d]_{\mathbb{C}^*\text{-action}} & Z \ar[d]^{\mathbb{C}^*\text{-action}}  \\
    \lambda^{-1}\cdot u=(\lambda^{-2} u_2,\lambda^{-4}u_4,\lambda^{-6}u_6,\lambda^{-8}u_8,\lambda^{-10}u_{10},\lambda^{-14}u_{14}) \ar[r]^{\hspace{3.9cm}\text{period mapping}}   & \lambda Z \\
}
\end{align}

From Definition \ref{DfMeroModularForm},
together with (\ref{SectionAssBundle}) and (\ref{QOplus}),
we have the assertion.
\end{proof}

The integer $\kappa_0=3$ in Table 1 is coming from this theorem.
Namely,
$(2\kappa_0, 4\kappa_0,6\kappa_0,8\kappa_0,10\kappa_0,14\kappa_0)=(6,12,18,24,30,42)$
is equal to the degrees of the group No.34 (see \cite{LT} Appendix D).

\subsection{Meromorphic modular forms of character ${\rm det}$ }

We will study the orbifold
$$
\mathbb{O}=[(\mathcal{D}^\circ)^*/(\Gamma \times \mathbb{C})]
$$
in order to construct modular forms for ${\rm det}$.

First, let us observe the action of $\Gamma $ on $\mathcal{D}^\circ$ precisely.
The action of $\Gamma$ on $\mathcal{D}^\circ$ is effective.
We set
$$
\mathfrak{H}_{\mathcal{D}^\circ}=\bigcup_{g \in \Gamma} \{[Z]\in\mathcal{D}^\circ  \mid g([Z])=[Z]\}.
$$
Also, letting $\Gamma_{[Z]}$ be the stabilize subgroup with respect to $[Z]\in \mathcal{D}^\circ$,
we set
$$
\mathfrak{S}_{\mathcal{D}^\circ}=\{[Z]\in \mathcal{D}^\circ \mid \Gamma_{[Z]} \text{ is neither } \{{\rm id}_{\Gamma}\} \text{ nor } \{{\rm id}_\Gamma,\sigma_\delta\} \text{ for } \delta\in \Delta({\bf A})  \}.
$$
According to Proposition \ref{PropKneser},
$\mathfrak{H}_{\mathcal{D}^\circ}$ is a countable union of reflection hypersurfaces
and $\mathfrak{S}_{\mathcal{D}^\circ}$ is a countable union of analytic subsets of codimension at least $2$.

Let $\mathfrak{H}_{Q} $ and $\mathfrak{S}_{Q}$ be the images of 
 $\mathfrak{H}_{\mathcal{D}^\circ} $ and $\mathfrak{S}_{\mathcal{D}^\circ}$ 
by the projection  $\mathcal{D}^\circ \rightarrow Q=\mathcal{D}^\circ/\Gamma$, respectively. 
From Proposition \ref{PropKneser},
it follows that
\begin{align}\label{Gamma'}
\Gamma':=\{ \gamma \in \Gamma  \mid \gamma \text{ is given by a product of reflections of even numbers }\}
=\{\gamma \in \Gamma  \mid {\rm det}(\gamma)=1\}.
\end{align} 
We set $Q_1=\mathcal{D}^\circ/\Gamma'$.
The action of $\Gamma'$ on $\mathcal{D}^\circ-\mathfrak{S}_{\mathcal{D}^\circ}$ is free.
Let us naturally define $\mathfrak{H}_{Q_1}$ and $\mathfrak{S}_{Q_1}$, respectively.
Recall that the period mapping $\Phi$ gives an identification $\mathcal{U}\simeq Q$ (see (\ref{PhiUIso})).
Set $\mathfrak{H}_{\mathcal{U}}=\Phi^{-1} (\mathfrak{H}_Q)$ and $\mathfrak{S}_{\mathcal{U}}=\Phi^{-1} (\mathfrak{S}_Q)$.
Then, 
$\mathfrak{H}_\mathcal{U}$ gives a divisor on the weighted projective space $\widehat{\mathcal{U}}$
and there exists a weighted homogeneous polynomial $\Delta_\mathcal{U}(u)\in \mathbb{C}[u_2,u_4,u_6,u_8,u_{10},u_{14}]$ such that
\begin{align}\label{H_U}
\mathfrak{H}_\mathcal{U}=\left\{[u]\in\widehat{\mathcal{U}}=\mathbb{P}(2,4,6,8,10,14) \mid \Delta_{\mathcal{U}}(u)=0\right\}.
\end{align}
We have the double covering $\mathcal{U}_1$ of $\mathcal{U}-\mathfrak{S}_\mathcal{U} $ branched along $\mathfrak{H}_\mathcal{U} - \mathfrak{S}_\mathcal{U}$: 
\begin{align}\label{U1}
\mathcal{U}_1=\{([u],s)\in (\mathcal{U}-\mathfrak{S}_\mathcal{U})\times \mathbb{C} \mid s^2=\Delta_\mathcal{U}(u)\}.
\end{align}

We can obtain the lift
$\Phi_{Q_1}: \mathcal{U}_1 \rightarrow Q_1-\mathfrak{S}_{Q_1}$
of $\Phi |_{\mathcal{U}-\mathfrak{S}_\mathcal{U}}$
so that $\Phi_{Q_1}$ is equivalent under the action of $\Gamma/\Gamma' \simeq \mathbb{Z}/2\mathbb{Z}.$
Also, $\Phi_{Q_1}$ is lifted to
$\Phi_{\mathcal{D^\circ}}:\mathcal{U}_\mathcal{D^\circ} \rightarrow \mathcal{D^\circ}-\mathfrak{S}_\mathcal{D^\circ}$,
which is equivalent under the action of $\Gamma$.
We can consider the pull-back $\mathcal{U}_{(\mathcal{D}^\circ)^*} \rightarrow \mathcal{U}_{\mathcal{D}^\circ}$
of the principal bundle $\mathcal{U}^* \rightarrow \mathcal{U}$ 
by the composition
$\mathcal{U}_{\mathcal{D}^\circ} \rightarrow \mathcal{U}_1 \rightarrow \mathcal{U}-\mathfrak{S}_\mathcal{U} \hookrightarrow \mathcal{U}.$
Then, we have the lifted period mapping
$$
\Phi_{(\mathcal{D}^\circ)^*}: \mathcal{U}_{(\mathcal{D}^\circ)^*} \simeq (\mathcal{D}^\circ)^* -\mathfrak{S}_{(\mathcal{D}^\circ)^*},
$$
where $\mathfrak{S}_{(\mathcal{D}^\circ)^*}$ is the preimage of $\mathfrak{S}_{\mathcal{D}^\circ}$ under the projection.

\begin{lem}
The lifted period mapping $\Phi_{(\mathcal{D}^\circ)^*}$ induces an isomorphism $[\Phi_{(\mathcal{D}^\circ)^*}]:[\mathcal{U}_{(\mathcal{D}^\circ)^*}/(\mathbb{C}^*\times \Gamma)] \simeq [((\mathcal{D}^\circ)^*-\mathfrak{S}_{\mathcal{D}^\circ})/(\mathbb{C}^*\times \Gamma)].$
\end{lem}

\begin{proof}
The proof is similar to \cite{Na} Section 3.3.
See the  diagram (\ref{OrbitDiagram}) also.
\begin{align}\label{OrbitDiagram}
  \xymatrix@C=50pt{
\mathcal{U}_{(\mathcal{D}^\circ)^*}\ar[r]^{\mathbb{C}^*\text{-bundle}}   &  \mathcal{U}_\mathcal{D^\circ}  \ar[r]^{\Phi_{\mathcal{D}^\circ}\hspace{5mm}} \ar[d]_{p_{\mathcal{D}^\circ}}^{\text{free}} & \mathcal{D}^\circ-\mathfrak{S}_{\mathcal{D}^\circ} \ar[d]^{\text{free}} &\ar[l]_{\hspace{1mm}\mathbb{C}^*\text{-bundle}} (\mathcal{D}^\circ)^*-\mathfrak{S}_{(\mathcal{D}^\circ)^*}  \\
   &  \mathcal{U}_1 \ar[r]^{\Phi_{Q_1}\hspace{5mm}}\ar[d]_{p_1}^{\mathbb{Z}/2\mathbb{Z}} & Q_1-\mathfrak{S}_{Q_1}\ar[d]^{\mathbb{Z}/2\mathbb{Z}}& \\
 & \mathcal{U}-\mathfrak{S}_\mathcal{U} \ar[r]^{\Phi |_{ \mathcal{U}-\mathfrak{S}_\mathcal{U}}\hspace{5mm}}\ar@{^{(}-_>}[d] &Q-\mathfrak{S}_{Q}\ar@{^{(}-_>}[d] & \\
\mathcal{U}^* \ar[r]^{\mathbb{C}^*\text{-bundle}}  &\mathcal{U} \ar[r]^{\Phi \hspace{5mm}}&Q=\mathcal{D}^\circ/\Gamma&\ar[l]_{\hspace{1mm}\mathbb{C}^*\text{-bundle}} (\mathcal{D}^\circ)^*/\Gamma \\
}
\end{align}
\end{proof}

\begin{prop}\label{PropPicard}
The Picard group
${\rm Pic}(\mathbb{O})  $ 
of the orbifold $\mathbb{O}$ is isomorphic to $ \mathbb{Z}\oplus (\mathbb{Z}/2\mathbb{Z}).$
\end{prop}

\begin{proof}
Set $\mathcal{V}_1=\left\{([u],s)\in \widehat{\mathcal{U} } \times \mathbb{C} \mid s^2=\Delta_\mathcal{U}(u) \right\}$.
Then, $\mathcal{U}_1$ of (\ref{U1}) satisfies $\mathcal{U}_1 \subset \mathcal{V}_1$ and ${\rm codim}(\mathcal{V}_1-\mathcal{U}_1)\geq 2$, from (\ref{UHartogus}).
Recall that $\Phi$ in (\ref{OrbitDiagram}) is extended to the identification (\ref{HatPhi}).
This $\widehat{\Phi}$ is lifted to $\widehat{\Phi}_{Q_1}: \mathcal{V}_1 \simeq \widehat {Q}_1$, which is equivalent under the $(\mathbb{Z}/2\mathbb{Z})$-action.
Here, $\widehat{Q}_1 $ is a double covering of $\widehat{\mathcal{D}^\circ/\Gamma}^{\bf L}.$

We consider the orbifold $\mathbb{V}_1=[\mathcal{V}_1/(\mathbb{Z}/2\mathbb{Z})]$ with the structure morphism $p_{\mathbb{V}_1}:\mathbb{V}_1 \rightarrow \widehat{\mathcal{U}}.$
Now, the Picard group ${\rm Pic}(\mathbb{V}_1)$ is generated by $\mathcal{O}_{\mathbb{V}_1}(1) := p_{\mathbb{V}_1}^* \mathcal{O}_{\widehat{\mathcal{U}}}(1)$ and the generator $g$ of $\mathbb{Z}/2\mathbb{Z}$.
We remark that this generator $g$  is corresponding to  $ {\rm det} \in {\rm Char}(\Gamma)$ (see Proposition \ref{PropKneser} and (\ref{Gamma'})).
The divisor $\{\Delta_\mathcal{U}(u)=0\}$ is corresponding to the reflection hypersurfaces for our lattice ${\bf A}$ via $\Phi$ in (\ref{OrbitDiagram}).
Therefore, 
any element of ${\rm Pic}(\mathbb{V}_1) $ is given by $\mathcal{O}_{\mathbb{V}_1}(1) ^{\otimes k} \otimes g^l$ for $k\in \mathbb{Z}$ and $l\in \mathbb{Z}/2\mathbb{Z}$.
By considering  $\left[\widehat{\Phi}_{Q_1}\right] : \mathbb{V}_1 \simeq \left[\widehat{Q}_1/(\mathbb{Z}/2\mathbb{Z})\right]$, we obtain 
${\rm Pic}\left( \left[\widehat{Q}_1/(\mathbb{Z}/2\mathbb{Z})\right]\right) \simeq \mathbb{Z} \oplus (\mathbb{Z}/2\mathbb{Z}).$

Recall that the action of $\Gamma'$ on $\mathcal{D}^\circ-\mathfrak{S}_{\mathcal{D}^\circ}$ is free.
From the fact that ${\rm codim}\left(\widehat{Q}_1-Q_1\right)\geq 2$, we have
\begin{align*}
{\rm Pic}(\mathbb{O}) = {\rm Pic}([(\mathcal{D}^\circ)^*/(\mathbb{C} \times \Gamma)]) = {\rm Pic}([Q_1/(\mathbb{Z}/2 \mathbb{Z})])\simeq 
{\rm Pic}\left(\left[\widehat{Q}_1/(\mathbb{Z}/2 \mathbb{Z})\right]\right)\simeq \mathbb{Z} \oplus (\mathbb{Z}/2\mathbb{Z}).
\end{align*}
\end{proof}

When we consider holomorphic sections of line bundles,
analytic subsets of codimension at least $2$ do not affect the results due to Hartogs's phenomenon.
So, in this section,
we shall omit such analytic sets.
Namely, from now on,
we will often omit such sets (like $\mathfrak{S}_\mathcal{U}$ or $\mathfrak{S}_{Q}).$

\begin{prop}\label{PropDeg}
The weight of $\Delta_\mathcal{U}(u)$ is equal to $98.$
\end{prop}

\begin{proof}
Since $\mathcal{U}$ is a Zariski open set in $\widehat{\mathcal{U}}=\mathbb{P}(2,4,6,8,10,14)$, the canonical bundle $\Omega_{\widehat{\mathcal{U}}}$ is calculated as 
$$
\Omega_\mathcal{U} \simeq \mathcal{O}_\mathcal{U}(-2-4-6-8-10-14)=\mathcal{O}_\mathcal{U}(-44).
$$
Let $p_1$ be the double covering $\mathcal{U}_1 \rightarrow \mathcal{U}$ branched along $\mathfrak{H}_\mathcal{U}$.
We  obtain the isomorphism $\Omega_{\mathcal{U}_1} \simeq p_1^* \Omega_\mathcal{U} \otimes \mathcal{O}_{\mathcal{U}_1} (\mathfrak{H}_{\mathcal{U}_1})$, by considering the holomorphic differential forms.
So, we have
\begin{align*}
\Omega_{[\mathcal{U}_1/(\mathbb{Z}/2\mathbb{Z})]} \simeq [p_1]^*\Omega_{\mathcal{U}} \otimes \mathcal{O}_{[\mathcal{U}_1/(\mathbb{Z}/2\mathbb{Z})]}(\mathfrak{H}_{\mathcal{U}_1}) \simeq [p_1]^*\mathcal{O}_{\mathcal{U}} (-44) \otimes \mathcal{O}_{[\mathcal{U}_1/(\mathbb{Z}/2\mathbb{Z})]}(\mathfrak{H}_{\mathcal{U}_1}).
\end{align*}

From the proof of Proposition \ref{PropPicard}, the orbifold $\mathbb{O}$ is equivalent to $[\mathcal{U}_1/(\mathbb{Z}/2\mathbb{Z})]$.
So,  ${\rm Pic}([\mathcal{U}_1/(\mathbb{Z}/2\mathbb{Z})])$ is isomorphic to $\mathbb{Z} \oplus (\mathbb{Z}/2\mathbb{Z})$.
Let $d$ be the weight of $\Delta_\mathcal{U}(u)$.
Then, we have $[p_1]^*\mathcal{O}_\mathcal{U} (d) \simeq \mathcal{O}_{[\mathcal{U}_1/(\mathbb{Z}/2\mathbb{Z})]} (2\mathfrak{H}_{\mathcal{U}_1})$.
This implies that the  direct summand $\mathbb{Z}/2\mathbb{Z}$ of ${\rm Pic}([\mathcal{U}_1/(\mathbb{Z}/2\mathbb{Z})])$ is generated by $[p_1]^*\mathcal{O}_{\mathcal{U}}(-d/2)\otimes \mathcal{O}_{[\mathcal{U}_1/(\mathbb{Z}/2\mathbb{Z})]}(\mathfrak{H}_{\mathcal{U}_1})$.
Here, we do not need to worry about whether $-d/2$ is an integer,
for all weights of $\widehat{\mathcal{U}}$ are even numbers.

Since $\mathcal{D}^\circ$ is a Zariski open set in a quadratic hypersurface in the projective space $\mathbb{P}^6(\mathbb{C})$,
we apply the adjunction formula to $\mathcal{D}^\circ$ and obtain
$
\Omega_{\mathcal{D}^\circ} \simeq \mathcal{O}_{\mathcal{D}^\circ} (7-2)=\mathcal{O}_{\mathcal{D}^\circ} (5)$,
where the weight is concordant with the $\mathbb{C}^*$-action indicated in (\ref{DiagramWeight}).
This implies that the canonical orbibundle $\Omega_\mathbb{O}$ is isomorphic to $\mathcal{O}_\mathbb{O} (5) \otimes {\rm det}$.

By summarizing the above properties,
we have
\begin{align*}
& \mathcal{O}_\mathbb{O} (5) \otimes {\rm det} \simeq \Omega_\mathbb{O} \simeq \Omega_{[\mathcal{U}_1/(\mathbb{Z}/2\mathbb{Z})]}\\
&\simeq [p_1]^*\mathcal{O}_\mathcal{U}(-44)\otimes \mathcal{O}_{[\mathcal{U}_1/(\mathbb{Z}/2\mathbb{Z})]}(\mathfrak{H}_{\mathcal{U}_1})\\
&\simeq [p_1]^*\mathcal{O}_\mathcal{U}\Big(-44+\frac{d}{2}\Big)\otimes \Big([p_1]^*\mathcal{O}_\mathcal{U}\Big(-\frac{d}{2}\Big)\otimes \mathcal{O}_{[\mathcal{U}_1/(\mathbb{Z}/2\mathbb{Z})]}(\mathfrak{H}_{\mathcal{U}_1} ) \Big)
\end{align*} 
and we obtain $d=98$.
\end{proof}

\begin{thm}\label{Thmdet}
(1) There exist holomorphic functions $s_7$ of weight $7$ and $s_{42} $ of weight $42$ on $(\mathcal{D}^\circ)^*$ such that
$$
s_7^2 = u_{14} , \quad \quad s_{42}^2 = d_{84}(u),
$$
where $d_{84}(u)$ is the polynomial studied in Lemma \ref{LemSingular}.

(2)
The holomorphic function $s_{49}=s_7 s_{42}$ on $(\mathcal{D}^\circ)^*$ gives a modular form of weight $49$ and character ${\rm det}$.
It holds $\mathcal{A}^\circ (\Gamma,{\rm det})=s_{49} \mathcal{A}^\circ (\Gamma,{\rm id})$.
\end{thm}

\begin{proof}
(1) 
By the argument in this section,
the divisor $\{[u]\in\mathcal{U} \mid \Delta_\mathcal{U}(u)=0\}$ corresponds to the union of reflection hyperplanes of $\{H_\delta \mid \delta\in \Delta({\bf A})\} - \mathcal{H}$.
Here, $H_\delta$ is a reflection hyperplane defined by $\delta\in \Delta({\bf A})$ and 
$\mathcal{H}$ is the arrangement of $\gamma H_0$  ($\gamma\in \Gamma$, $H_0=\{\xi_7=0\}$)   as in Section 1.1.
Recalling the observation of degenerations of our lattice polarized $K3$ surfaces in Lemma \ref{LemSingular}
and the meaning of the canonical form (\ref{SK3Can}),
we have
\begin{align}\label{DeltaDecomp}
\{[u]\in \mathcal{U} \mid d_{84}(u)=0\} \cup \{[u] \in \mathcal{U} \mid u_{14}=0\}\subset \{[u]\in \mathcal{U} \mid \Delta_\mathcal{U}(u)=0\}.
\end{align}
The inclusion (\ref{DeltaDecomp}) means that  $u_{14} d_{84}(u)$ divides $\Delta_\mathcal{U}(u)$ in $\mathbb{C}[u_2,u_4,u_6,u_8,u_{10},u_{14}].$
On the other hand,
Proposition \ref{PropDeg} says that $\Delta_\mathcal{U}(u)$ is of weight $98$.
Thus, we have the irreducible decomposition
\begin{align}\label{Decomposition}
\Delta_\mathcal{U}(u) = {\rm const } \cdot u_{14} d_{84}(u).
\end{align}
Since the double covering $p_1:\mathcal{U}_1 \rightarrow \mathcal{U}$ in the diagram (\ref{OrbitDiagram}) is branched along the divisor $\{[u]\in \mathcal{U} \mid \Delta_\mathcal{U}(u)=0\}$,
(\ref{Decomposition}) implies that there is a holomorphic function $s_7$ ($s_{42}$, resp.) on $(\mathcal{D}^\circ)^*$ satisfying
$s_7^2=u_{14}$ ($s_{42}^2=d_{84}(u)$, resp.).

(2) 
Recall that
 ${\rm det}\in {\rm Char}(\Gamma)$ is coming from the action of $\Gamma/\Gamma'\simeq \mathbb{Z}/2\mathbb{Z}$ which defines the double covering $p_1:\mathcal{U}_1\rightarrow \mathcal{U}$.
By the Definition \ref{DfMeroModularForm} and the meaning of  (\ref{U1}), 
 every modular form of character ${\rm det}$ vanishes on the preimage of 
 $\mathfrak{H}_{\mathcal{D}^\circ}$ by the canonical projection $(\mathcal{D}^\circ)^* \rightarrow \mathcal{D}^\circ$.
 Since this $p_1$ is branched along the divisor $\{\Delta_\mathcal{U}(u)=0\},$ (\ref{DeltaDecomp}) and (\ref{Decomposition}) show that 
every modular form of character ${\rm det}$ is given by a product of $s_{49}:=s_7 s_{42}$ and a modular form of character ${\rm id}.$
\end{proof}

\begin{rem}
The Weierstrass equation of (\ref{SK3}) is essential for our purpose.

We have another expression of elliptic $K3$ surfaces with  singular fibres of type (\ref{SingularFibre})
given by the equation
 \begin{align}\label{SOp}
z_1^2=y_1^3+(b_1 x_1^5 w_1+ b_4 x_1^4 w_1^4+ b_7 x_1^3 w_1^7) y_1+(b_0 x_1^8 +b_3 x_1^7 w_1^3 + b_6 x_1^6 w_1^6 + b_9 x_1^5 w_1^9)
\end{align}
 of weight $24$, where $(x:y:z:w)\in \mathbb{P}(3,8,12,1)$.
However, 
the expression (\ref{SOp}) is not appropriate
to construct  modular forms.
Although (\ref{SOp}) can be birationally transformed to the Weierstrass form (\ref{SK3}), 
it is impossible to construct correct modular forms from the parameters in (\ref{SOp}).
It seems that we can obtain modular forms of weight $1,3,4,6,7,9$ on $(\mathcal{D}^\circ)'$ from (\ref{SOp}),
which is a complement of an arrangement corresponding to the condition $b_0 =0.$
However, we can see that this arrangement does not satisfy the condition of Theorem \ref{ThmLooijenga}.
So, the expression (\ref{SOp}) induces an erroneous use of the theory of the  Looijenga compactifications.

For example,
it seems that
we can obtain a modular form of weight $7$ coming from the parameter $b_7$.
We can see that the zero set of this modular form coincides with  the arrangement $\mathcal{H}$ in Lemma \ref{LemArrangement}.
However, as stated in \cite{L} (see also \cite{LBall} Section 6),
if an arrangement gives the zero set of a modular form,
then this does not satisfy the condition of Theorem \ref{ThmLooijenga}.
This contradicts to Lemma \ref{LemArrangement}.

Thus,  
our expression of  (\ref{SK3}) is suitable for the theory of  the Looijenga compactifications
and  effective to construct modular forms. 
\end{rem}

\begin{rem}\label{RemKappa}
Theorem \ref{Thmdet}  supports the relation between our sequence of the families and complex reflection groups.
The weight $42$ of $s_{42} $ is coming from the discriminant of the right hand side of (\ref{SK3Can}).
We note that $126=42 \kappa_0$, 
where $\kappa_0=3$ as in Table 1, 
is equal to the number of reflections of order $2$ for the group No.34 (see \cite{LT} Appendix D).

Such a phenomenon occurs in each case for $\mathfrak{F}_j$  ($j\in \{1,2,3\}$).
In the case for $j=1$ ($2,3$, resp.),
there is a holomorphic function of weight $45$ ($30$, $15$, resp.) coming from the discriminant of the elliptic $K3$ surfaces of \cite{Na} (\cite{CD}, \cite{NHilb}, resp.).
Then,  $45=45\kappa_1$ ($60=30\kappa_2$, $15=15\kappa_3$, resp.) is equal to the number of reflections of order $2$ for the group No.33 (No.31, No.23, resp.).  
Each of these holomorphic functions gives a factor of a modular form of a non-trivial character.
Also, there are explicit expressions of them in terms of the theta functions in Table 4.
\end{rem}

\section{Transcendental lattice for  family $\mathfrak{G}_0$ of Kummer-like surfaces with Picard number $15$}

In Section 2,
we determined the lattice structure for the family $\mathfrak{F}_0$ via a natural consideration of singular fibres for elliptic surfaces.
Also, in fact, the lattices for the subfamilies  $\mathfrak{F}_1,\mathfrak{F}_2,\mathfrak{F}_3$ in Proposition \ref{PropLatticeF} can be determined in a similar  way.

However, 
as for the families $\mathfrak{G}_j$ $(j\in \{1,2,3\})$ and $\mathfrak{G}_1'$ of the Kummer-like surfaces,
it is much harder to determine their lattices correctly.
For example,
\begin{itemize}
\item
The family $\mathfrak{G}_2$ is the family of the Kummer surfaces for principally polarized Abelian surfaces.
For a precise study for the lattice structure of the Kummer surfaces,
Nikulin \cite{NiKummer} introduces a particular lattice which is called the Kummer lattice.
Also, Morrison \cite{Mo} studies an interesting viewpoint called the Shioda-Inose structure
for $K3$ surfaces whose Picard numbers are greater than $17$.

\item
In order to determine the transcendental lattice for the family $\mathfrak{G}_1'$ of Kummer-like surfaces with Picard number $16$,
Matsumoto-Sasaki-Yoshida \cite{MSY} (see also \cite{Y}) study hypergeometric integrals of type $(3,6)$ 
and calculate intersection numbers of the chambers coming from the configuration of six lines on $\mathbb{P}^2(\mathbb{C})$ by applying a delicate technique of twisted homologies.

\item
For the transcendental lattice for the family $\mathfrak{G}_1$ of Kummer-like surfaces with Picard number $16$,
Shiga and the author \cite{NS2} have a geometric construction of $2$-cycles on a generic member of $\mathfrak{G}_1$
taking into account the fact that  a generic member of $\mathfrak{G}_1$ is a double covering of that of  $\mathfrak{F}_1$. 
This construction is also based on  hard calculations of local monodromies for elliptic surfaces. 
\end{itemize}

 In this section,
 we will determine the transcendental lattice for the family $\mathfrak{G}_0$.
It is a non-trivial problem to determine it.
If there were a double covering $S_a \rightarrow K_a$ for generic members of $\mathfrak{F}_0$ and $\mathfrak{G}_0$, 
we could calculate the lattice for $\mathfrak{G}_0$ from the lattice for  $\mathfrak{F}_0$ in Theorem \ref{TheoremLattice}
via a technique of Nikulin \cite{Ni} Section 2.
However, in practice, we can prove that there is no such a double covering for generic $a\in \mathfrak{A}_0$.
This proof can be given  in a similar way to the proof of \cite{NS2} Theorem 6.3.

Now,
let us remark  the fact that our family $\mathfrak{G}_0$ naturally contains the families $\mathfrak{G}_1$ and $\mathfrak{G}_1'.$
We will determine the transcendental lattice for $\mathfrak{G}_0$ based on this fact.
The main result in this section is indebted to  heavy calculations for  $\mathfrak{G}_1'$ in \cite{MSY} and those for  $\mathfrak{G}_1$  in  \cite{NS2}.

 The following result for even lattices is necessary for our proof.
 
 \begin{lem}
 (\cite{Mo}, Corollary 2.10)\label{LemEmb(2,t)}
 Suppose $12\leq \rho \leq 20$.
  Let ${\bf T}$ be an even lattice of signature $(2,20-\rho).$
  Then, the primitive embedding ${\bf T} \hookrightarrow L_{K3}$ is unique up to isometry.
 \end{lem}

The following theorem is the main result of this section.

\begin{thm}\label{ThmTrK}
The transcendental lattice of a generic member of $\mathfrak{G}_0$ is given by the intersection matrix
$U(2)\oplus U(2) \oplus \begin{pmatrix} -2 & 0 & 1 \\ 0 & -2 & 1 \\ 1 & 1 & -4 \end{pmatrix}$.
\end{thm}

\begin{proof}
We identify the $K3$ lattice $L_{K3}=U^{\oplus 3} \oplus E_8(-1)^{\oplus 2}$ with the $2$-homology group of $K3$ surfaces.

Let $e_j,f_j$ $(j\in\{1,2,3\})$ be elements of $U^{\oplus 3}$ satisfying $(e_j\cdot e_k)=(f_j\cdot f_k)=0$ and $(e_j\cdot f_k)=\delta_{j,k}$.
Let $p_j,q_j$ $(j\in\{1,\ldots,8\})$ be elements of  $E_8(-1)^{\oplus 2}$ with the intersection numbers defined by the Dynkin diagram in Figure 2.
Also, put $\nu_j=p_j+q_j$.
\begin{figure}[h]
\center
\includegraphics[scale=0.5, bb=150 320 500 520]{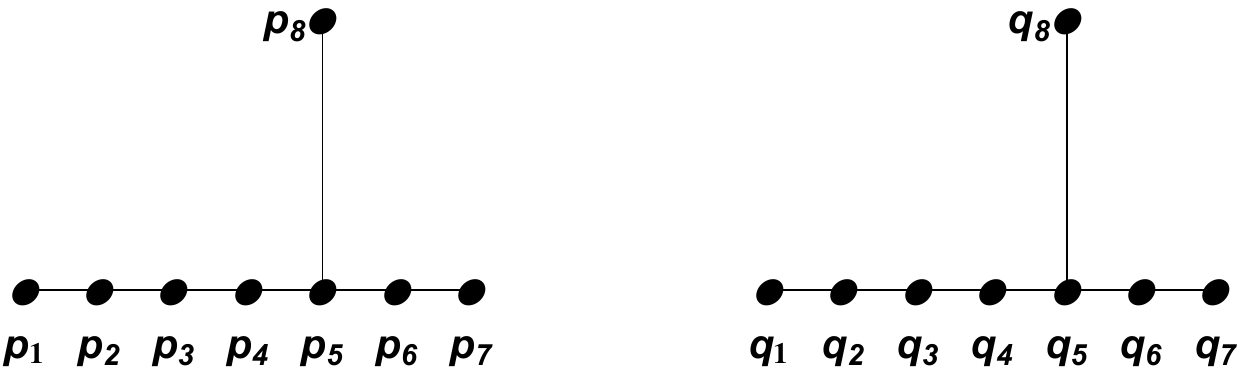}
\caption{Basis of $E_8(-1)^{\oplus 2}$}
\end{figure}

By reference to  \cite{Ma},
we put
\begin{align}\label{U(2)U(2)}
\begin{cases}
& \lambda_1=-\nu_5 +\nu_7+2(e_1+f_1), \quad \mu_1=-\nu_4,\\
& \lambda_2=\nu_7 +\nu_8 +2(e_1+e_2+e_3+f_3),\quad
 \mu_2=\nu_6.
\end{cases} 
\end{align}
Then, $\{\lambda_1,\mu_1,\lambda_2,\mu_2\} $ gives a basis of the lattice $U(2)^{\oplus 2}$.
This basis defines a primitive embedding $U(2)^{\oplus 2} \hookrightarrow L_{K3}$. 

Let us recall Proposition \ref{PropLatticeG}.
By extending the basis (\ref{U(2)U(2)}),
we have a primitive embedding of the transcendental lattice ${\bf B}_1 = U(2)^{\oplus 2} \oplus A_2(-2)$ for the family $\mathfrak{G}_1$  into $L_{K3}$ given by 
\begin{align}\label{B1}
{\bf B}_1= \langle \lambda_1,\mu_1,\lambda_2,\mu_2, \nu_1,\nu_2 \rangle_\mathbb{Z},
\end{align}
where the transcendental lattice ${\bf B}_2 = U(2)^{\oplus 2} \oplus A_1(-2)$ for the subfamily $\mathfrak{G}_2$ is a primitive sublattice of $L_{K3}$ explicitly given by
\begin{align}\label{B2}
{\bf B}_2 =\langle \lambda_1,\mu_1,\lambda_2,\mu_2, \nu_1 \rangle_\mathbb{Z}.
\end{align}
According to Lemma \ref{LemEmb(2,t)},
this embedding is unique up to isometry.
So, we can fix the embedding given by (\ref{B1}) and (\ref{B2})  without loss of generality.

We remark that the family  $\mathfrak{G}_1'$ also contains  $\mathfrak{G}_2$ as a subfamily.
Hence, its transcendental lattice ${\bf B}_1'=U(2)^{\oplus 2} \oplus A_1(-1)^{\oplus 2}$
should be an extension of the  lattice ${\bf B}_2$ of (\ref{B2}).
So, we have the following explicit  basis: 
\begin{align} \label{B1'}
{\bf B}_1'=\langle \lambda_1,\mu_1,\lambda_2,\mu_2, p_1,q_1 \rangle_\mathbb{Z}.
\end{align}
We note that the expression  (\ref{B1'})  is guaranteed by  the fact that
the  lattice ${\bf B}_2$ of (\ref{B2}) is invariant under the involution  given by interchanging two $A_1(-1)$ summands of the lattice ${\bf B_1'}$
(see \cite{MSY}; see also  \cite{Y} Chapter IX).

Our family $\mathfrak{G}_0$ contains $\mathfrak{G}_1$ and $\mathfrak{G}_1'$.
So, the transcendental lattice for a generic member of $\mathfrak{G}_0$ 
is  a primitive lattice of $L_{K3}$, of rank $7$ and  given by the basis which is an extension of  (\ref{B1}) and (\ref{B1'}).
Such a lattice is given  by the explicit basis
$
\langle \lambda_1,\mu_1,\lambda_2,\mu_2, p_1,q_1,\nu_2 \rangle_\mathbb{Z}$
whose intersection matrix is
$
U(2)^{\oplus 2}  \oplus \begin{pmatrix} -2 & 0 & 1 \\ 0 & -2 & 1 \\ 1 & 1 & -4 \end{pmatrix}.
$
\end{proof}

Since we have a double covering $K_a\rightarrow S_a$ for generic members of $\mathfrak{G}_0$ and $\mathfrak{F}_0$,
we can testify the correctness of  Theorem \ref{ThmTrK}.
Namely, 
when ${\rm Tr}(K_a)$ is given,
we can calculate the intersection matrix of ${\rm Tr}(S_a)$. 
Let $\Lambda_0=({\rm Tr}(K_a) \otimes \mathbb{Q})\cap (U^{\oplus 3} \oplus \frac{1}{2} E_8(-2) )$,
where $E_8(-2)$ is the lattice generated by $\nu_1,\ldots,\nu_8$,
${\rm Tr}(S_a)$ is isometric to the lattice $\Lambda_0(2)$
(see \cite{Ni} Section 2).
In our case, $\Lambda_0(2)$ is given by the direct sum of
$U^{\oplus 2}$ and $\langle \nu_1/2, q_1 ,\nu_2/2 \rangle_\mathbb{Z}(2)$.
This is isometric to
$$
U^{\oplus 2} \oplus \begin{pmatrix} -1& -1& 1/2 \\  -1& -2 & 1/2 \\  1/2 & 1/2 & -1  \end{pmatrix}(2) \simeq U^{\oplus 2} \oplus \begin{pmatrix} -2&-2& 1\\  -2&-4& 1\\ 1&1&-2 \end{pmatrix} \simeq U^{\oplus 2} \oplus \begin{pmatrix} -2&1&0 \\ 1&-2&0 \\ 0&0&-2 \end{pmatrix}={\bf A}.
$$
This is concordant with Theorem \ref{TheoremLattice}.

\section*{Acknowledgment}
The author would like to thank  Professor Manabu Oura and Professor Jiro Sekiguchi for valuable suggestions from the viewpoint of complex reflection groups.
He also appreciates the reviewer's helpful comments to improve the manuscript.
This work is supported by JSPS Grant-in-Aid for Scientific Research (22K03226, 18K13383),
JST FOREST Program (JPMJFR2235)
and 
MEXT LEADER.

\begin{center}
\hspace{10cm}\textit{Atsuhira  Nagano}\\
\hspace{10cm}\textit{Faculty of Mathematics and Physics}\\
 \hspace{10cm} \textit{Institute of Science and Engineering}\\
\hspace{10cm}\textit{Kanazawa University}\\
\hspace{10cm}\textit{Kakuma, Kanazawa, Ishikawa}\\
\hspace{10cm}\textit{920-1192, Japan}\\
 \hspace{10cm}\textit{(E-mail: atsuhira.nagano@gmail.com)}
  \end{center}

\end{document}